\documentclass[twoside]{article}

\usepackage[accepted]{aistats2022}
\usepackage[round]{natbib}

\usepackage{microtype}
\usepackage{graphicx}
\usepackage{booktabs} 
\usepackage{makecell}

\usepackage{amsmath,amssymb,dsfont,empheq}
\usepackage{algorithm}
\usepackage{algpseudocode}
\usepackage{syntonly,bm,wrapfig}
\usepackage{xcolor,cases}

\usepackage{hyperref}

\makeatletter
\def\thanks#1{\protected@xdef\@thanks{\@thanks
        \protect\footnotetext{#1}}}
\makeatother

\input{mysymbol.sty}

\DeclareMathOperator*{\argmin}{arg\,min}

\newcommand{\EE}{\mathbb{E}}
\DeclarePairedDelimiter{\dotp}{\langle}{\rangle}

\DeclareMathOperator{\Var}{Var}

\def \bbeta {{\boldsymbol \beta}}

\newcommand\numberthis{\addtocounter{equation}{1}\tag{\theequation}}

\newtheorem{theorem}{Theorem}
\newtheorem{lemma}[theorem]{Lemma}
\newtheorem{proposition}[theorem]{Proposition}

\newtheorem{innercustomthm}{Theorem}
\newenvironment{customthm}[1]
  {\renewcommand\theinnercustomthm{#1}\innercustomthm}
  {\endinnercustomthm}

\newcommand{\FullTitle}{A Single-Timescale Method for Stochastic Bilevel Optimization}

\begin{document}
\twocolumn[

\aistatstitle{\FullTitle}

\aistatsauthor{ Tianyi Chen$^{\star}$ \And Yuejiao Sun$^{\dagger}$  \And  Quan Xiao$^{\star}$ \And  Wotao Yin$^{\dagger}$} 
\vspace{0.3cm}
\aistatsaddress{ $^{\star}$Rensselaer Polytechnic Institute ~~~~~~~~\And~~~~~~~~  $^{\dagger}$University of California, Los Angeles}
]

\begin{abstract}%
\vspace{-0.3cm}
Stochastic bilevel optimization generalizes the classic stochastic optimization from the minimization of a single objective to the minimization of an objective function that depends on the solution of another optimization problem. 
Recently, bilevel optimization is regaining popularity in emerging machine learning applications such as hyper-parameter optimization and model-agnostic meta learning. 
To solve this class of optimization problems, existing methods require either double-loop or two-timescale updates, which are sometimes less efficient. This paper develops a new optimization method for a class of stochastic bilevel problems that we term Single-Timescale stochAstic BiLevEl optimization (\textbf{STABLE}) method. 
STABLE runs in a single loop fashion, and uses a single-timescale update with a fixed batch size. 
To achieve an $\epsilon$-stationary point of the bilevel problem, STABLE requires ${\cal O}(\epsilon^{-2})$ samples in total; and to achieve an $\epsilon$-optimal solution in the strongly convex case, STABLE requires ${\cal O}(\epsilon^{-1})$ samples. To the best of our knowledge, when STABLE was proposed, it is the \emph{first} bilevel optimization algorithm achieving the same order of sample complexity as SGD for single-level stochastic optimization. 
\end{abstract}
\vspace{-0.3cm}

\section{Introduction}

In this paper, we consider solving the stochastic optimization problems of the following form
\begin{subequations}\label{opt0}
	\begin{align}
	&\min_{x\in \mathcal{X}}~~~F(x):=\mathbb{E}_{\xi}\left[f\left(x, y^*(x);\xi\right)\right] ~~~~~~\, {\rm (upper)}\\
	&~{\rm s. t.}~~~~~y^*(x)\in \argmin_{y\in \mathbb{R}^{d_y}}~\mathbb{E}_{\phi}[g(x, y;\phi)] ~~~{\rm (lower)}\label{opt0-low}
\end{align}
\end{subequations}
where $f$ and $g$ are differentiable functions; $\xi$ and $\phi$ are random variables; and $\mathcal{X}\subset\mathbb{R}^d$ is closed and convex set. The problem \eqref{opt0} is often referred to as the stochastic \emph{bilevel} problem, where the upper-level optimization problem depends on the solution of the lower-level optimization over $y\in  \mathbb{R}^{d_y}$, denoted as $y^*(x)$, which depends on the value of upper-level variable $x\in \mathcal{X}$.

 Bilevel optimization can be viewed as a generalization of the classic two-stage stochastic programming \citep{shapiro2009}, in which the upper-level objective function depends on the optimal lower-level objective value rather than the lower-level solution.
Earlier works have studied applications in portfolio management and game theory \citep{stackelberg}; see two recent surveys \citep{dempe2020bilevel,liu2021investigating}. 
Recently, bilevel optimization has gained growing popularity in a number of machine learning applications such as meta-learning \citep{rajeswaran2019meta}, reinforcement learning \citep{konda1999actor,hong2020ac}, hyper-parameter optimization \citep{pedregosa2016hyperparameter,franceschi2018bilevel}, continual learning \citep{borsos2020coresets}, and image processing \citep{kunisch2013bilevel}. 
In some of these applications, when the lower-level problem admits a closed-form solution, bilevel optimization also reduces to the recently studied stochastic compositional optimization \citep{wang2017mp,ghadimi2020jopt,chen2020scsc}. 

Unlike single-level stochastic problems, algorithms tailored for solving bilevel stochastic problems are much less explored. This is partially because solving this class of problems via traditional optimization techniques faces a number of challenges. A key difficulty due to the nested structure is 
that (stochastic) gradient, a basic element in continuous optimization machinery, is prohibitively expensive or even impossible to compute. As we will show later, since computing an unbiased stochastic gradient of $F(x)$ requires solving the lower-level problem once, running stochastic gradient descent (SGD) on the upper-level problem essentially results in a double-loop algorithm which uses an iterative algorithm to solve the lower-level problem thousands or even millions of times.

 \begin{table*}[t]
	\begin{center}
	   \caption{Sample complexity of several state-of-the-art algorithms (BSA in \citep{ghadimi2018bilevel}, TTSA in \citep{hong2020ac}, stocBiO in \citep{ji2020provably}) to achieve an $\epsilon$-stationary point of $F(x)$ in the nonconvex setting and an $\epsilon$-optimal solution of $F(x)$ in the strongly convex setting; the notation $\widetilde{\cal O}(\cdot)$ hides logarithmic terms of $\epsilon^{-1}$.}
	   \vspace{0.1cm}
		\begin{tabular}{c||c|c | c|c }
			\hline
			\hline
 	& \textbf{STABLE}  & \textbf{BSA} &  \textbf{TTSA}  & \textbf{stocBiO}   \\ \hline
			\textbf{batch size} & ${\cal O}(1)$   & $\widetilde{\cal O}(1)$   & $\widetilde{\cal O}(1)$  & $\widetilde{\cal O}(\epsilon^{-1})$      \\ \hline
		 \textbf{\# of loops} & Single  & Double& Single   & Double      \\ \hline
		 \makecell{\textbf{\# of samples} \\  \textbf{(nonconvex)}}   & \makecell{ ${\cal O}(\epsilon^{-2})$ in $\xi$ \\ ${\cal O}(\epsilon^{-2})$ in $\phi$} & \makecell{ ${\cal O}(\epsilon^{-2})$ in $\xi$ \\ $\widetilde{\cal O}(\epsilon^{-3})$ in $\phi$}  & \makecell{ ${\cal O}(\epsilon^{-2.5})$ in $\xi$ \\ $\widetilde{\cal O}(\epsilon^{-2.5})$ in $\phi$}& \makecell{ ${\cal O}(\epsilon^{-2})$ in $\xi$ \\ $\widetilde{\cal O}(\epsilon^{-2})$ in $\phi$}  \\ \hline
		 		\makecell{ \textbf{\# of samples}\\ \textbf{(strongly convex)}}& \makecell{ ${\cal O}(\epsilon^{-1})$ in $\xi$ \\ ${\cal O}(\epsilon^{-1})$ in $\phi$} &  \makecell{${\cal O}(\epsilon^{-1})$ in $\xi$\\ $\widetilde{\cal O}(\epsilon^{-2})$ in $\phi$}  & \makecell{ ${\cal O}(\epsilon^{-1.5})$ in $\xi$ \\ $\widetilde{\cal O}(\epsilon^{-1.5})$ in $\phi$}&  /    \\ \hline
		 \textbf{complexity of $y$ update} & ${\cal O}(d_y^3)$  & ${\cal O}(d_y^2)$& ${\cal O}(d_y^2)$   & ${\cal O}(d_y^2)$      \\ \hline 		
		 	\hline
		\end{tabular}
	\end{center}
			\label{table:comp1}
			\vspace{-0.2cm}
\end{table*}

\vspace{-0.1cm} 
\subsection{Prior art}
\vspace{-0.1cm}
To put our work in context, we review prior art that we group in the following two categories.

\noindent\textbf{Bilevel optimization.} 
Bilevel optimization has a long history in operations research, where the lower level problem is served as the constraint of the upper level problem \citep{bracken1973mathematical,ye1995optimality,vicente1994bilevel,colson2007overview}.
Many recent efforts have been made to solve the bilevel optimization problems.
One successful approach is to reformulate the bilevel problem as a single-level problem by replacing the lower-level problem by its optimality conditions~\citep{colson2007overview,kunapuli2008}. Recently, gradient-based first-order methods for bilevel optimization have gained popularity, where the idea is to iteratively approximate the (stochastic) gradient of the upper-level problem either in forward or backward manner \citep{pedregosa2016hyperparameter,sabach2017jopt,franceschi2018bilevel,shaban2019truncated,grazzi2020iteration}. While most of these works assume the unique solution of the lower-level problem, cases where this assumption does not hold have been tackled in the recent work  \citep{liu2020icml}.
All these algorithms have excellent empirical performance, but many of them either provide no theoretical guarantees or only focus on the asymptotic performance analysis.  

The non-asymptotic analysis of bilevel optimization algorithms has been recently studied in some \emph{pioneering} works, e.g., \citep{ghadimi2018bilevel,hong2020ac,ji2020provably}, just to name a few. In both \citep{ghadimi2018bilevel,ji2020provably}, bilevel stochastic optimization algorithms have been developed that run in a double-loop manner.
To achieve an $\epsilon$-stationary point, they only need the sample complexity ${\cal O}(\epsilon^{-2})$ that is comparable to the complexity of SGD for the single-level case. 
Recently, a single-loop two-timescale stochastic approximation algorithm has been developed in \citep{hong2020ac} for the bilevel problem \eqref{opt0}. Due to the nature of two-timescale update, it incurs the sub-optimal sample complexity ${\cal O}(\epsilon^{-2.5})$. Therefore, the existing single-loop solvers for bilevel problems are significantly slower than those for problems without bilevel compositions, but otherwise share many structures and properties. 

\noindent\textbf{Concurrent work.}
After our STABLE was developed and released, its rate of convergence was improved to ${\cal O}(\epsilon^{-1.5})$ by momentum accelerations in \citep{khanduri2021momentum,guo2021randomized,yang2021provably}. The adaptive gradient variant has been studied in \citep{huang2021biadam}. Besides, a tighter analysis for alternating stochastic gradient descent (ALSET) method was proposed in \citep{chen2021closing}. The contributions compared to ALSET are: (a) ALSET uses SGD on the lower level but STABLE has a correction term, so STABLE has a reduced stochastic oracle complexity; (b) STABLE can handle the \emph{constrained} upper-level problem using {\em Moreau envelop}.

\noindent\textbf{Stochastic compositional optimization.}
When the lower-level problem in \eqref{opt0-low} admits a smooth closed-form solution, the bilevel problem \eqref{opt0} reduces to stochastic compositional optimization.
Popular approaches tackling this class of problems use two sequences of variables being updated in two different time scales \citep{wang2017mp,wang2017jmlr}.
However, the complexity of \citep{wang2017mp} and \citep{wang2017jmlr} is worse than ${\cal O}(\epsilon^{-2})$ of SGD for the non-compositional case.  
Building upon recent variance-reduction techniques,  
variance-reduced methods have been developed to solve a special class of the stochastic compositional problem with the \emph{finite-sum structure}, e.g., \citep{lian2017aistats,zhang2019nips}, but they usually operate in a double-loop manner. Other
related compositional algorithms also include \citep{tran2020stochastic,hu2020}.

While most of existing algorithms rely on either two-timescale or double-loop updates, the single-timescale  single-loop approaches have been recently developed in \citep{ghadimi2020jopt,chen2020scsc}, which achieve the sample complexity ${\cal O}(\epsilon^{-2})$. These encouraging recent results imply that solving stochastic compositional optimization is nearly as easy as solving stochastic optimization.

Our work is also related to the stochastic min-max optimization; see e.g., \citep{daskalakis2018limit,luo2020stochastic,rafique2018non,mokhtari2020unified,lin2020gradient,nouiehed2019solving}.
However, whether the techniques used in compositional and min-max optimization permeate to solving more challenging bilevel problems remains unknown. This paper is devoted to answering this question.

\vspace{-0.2cm}
\subsection{Our contributions}
\vspace{-0.1cm}
To this end, this paper aims to develop a \emph{single-loop single-timescale} stochastic algorithm, which, for the class of smooth bilevel problems, can match the sample complexity of SGD for single-level stochastic optimization problems. In the context of existing methods, our contributions can be summarized as follows. 
 
\noindent\textbf{C1)}  We develop a new stochastic gradient estimator tailored for a certain class of stochastic bilevel problems, which is motivated by an ODE analysis for the corresponding continuous-time deterministic problems. Our new stochastic bilevel gradient estimator is flexible to combine with any existing stochastic optimization algorithms for the single-level problems, and solve this class of stochastic bilevel problems as sample-efficient as single-level problems. 
  
\noindent\textbf{C2)} When we combine this stochastic gradient estimator with SGD for the upper-level update, we term it as the Single-Timescale stochAstic BiLevEl optimization (\textbf{STABLE}) method. 
	In the nonconvex case, to achieve $\epsilon$-stationary point of \eqref{opt0}, STABLE only requires ${\cal O}(\epsilon^{-2})$ samples in total. In the strongly convex case, to achieve $\epsilon$-optimal solution of \eqref{opt0}, STABLE only requires ${\cal O}(\epsilon^{-1})$ samples. This is achieved by designing a new Lyapunov function. To the best of our knowledge, when STABLE was proposed, it is the \emph{first} bilevel algorithm achieving the order of sample complexity as SGD. See the sample complexity of state-of-the-art algorithms in Table \ref{table:comp1}.

 \noindent\textbf{Trade-off and limitations.} 
 While our new bilevel algorithm significantly improves the sample complexity of existing algorithms, it pays the price of additional computation per iteration. Specifically, in order to better estimate the stochastic bilevel gradient, a matrix inversion and an eigenvalue truncation are needed per iteration, which cost ${\cal O}(d^3)$ computation for a $d\times d$  matrix. In contrast, some of recent works \citep{ghadimi2018bilevel,hong2020ac,ji2020provably} reduce matrix inversion to more efficient computations of matrix-vector products, which cost ${\cal O}(d^2)$ computation per iteration. Therefore, our algorithm is preferable in the regime where the sampling is more costly than computation or the dimension $d$ is relatively small. 
 
\section{A Single-timescale Optimization Method for Bilevel Problems}
\label{sec.scg}
 
In this section, we will first provide background of bilevel problems, and then present our  stochastic bilevel gradient method, followed by an ODE analysis to highlight the intuition of our design.  

\subsection{Preliminaries}
We use $\|\cdot \|$ to denote the $\ell_2$ norm for vectors and Frobenius norm for matrices.
We use ${\cal F}^k$ to denote the collection of random variables, i.e., ${\cal F}^k:=\left\{\phi^0, \ldots, \phi^{k-1}, \xi^0, \ldots, \xi^{k-1}\right\}$. We define the deterministic version of \eqref{opt0} without constraint on $\mathcal{X}$ as
\begin{align}\label{opt0-deter}
&\min_{x\in \mathbb{R}^d}F(x):=f\left(x, y^*(x)\right)\nonumber\\
&~{\rm s.t.}~y^*(x)\in \argmin_{y\in \mathbb{R}^{d_y}} g(x, y)
\end{align}
where the functions are defined as $g(x, y):=\mathbb{E}_{\phi}[g(x, y;\phi)]$ and $f(x, y):=\mathbb{E}_{\xi}[f(x, y;\xi)]$.

We also define $\nabla_{yy}^2 g\left(x, y\right)$ as the Hessian matrix of $g$ with respect to $y$ and define $  \nabla_{xy}^2g\left(x, y\right)$ as
\begin{align*}
  \nabla_{xy}^2g\left(x, y\right) :=    \begin{bmatrix}
    \frac{\partial^2}{\partial x_1\partial y_1 }g\left(x, y\right) & \cdots & \frac{\partial^2}{\partial x_1\partial y_{d_y}}g\left(x, y\right)\\
    & \cdots & \\
    \frac{\partial^2}{\partial x_d\partial y_1 }g\left(x, y\right) & \cdots & \frac{\partial^2}{ \partial x_d\partial y_{d_y}}g\left(x, y\right)
  \end{bmatrix}.
\end{align*}


We make the following standard assumptions that are commonly used in stochastic bilevel optimization literature \citep{ghadimi2018bilevel,hong2020ac,ji2020provably,khanduri2021momentum,guo2021randomized}. 

\vspace{0.1cm}
\noindent\textbf{Assumption 1 (Lipschitz continuity).}
\emph{
For any $x$, $\nabla_x f(x,\cdot)$, $\nabla_y f(x,\cdot)$, $\nabla_y g(x,y)$, $\nabla^2_{xy}g(x,\cdot;\phi)$, $\nabla^2_{yy}g(x,\cdot;\phi)$ are $L_{f_x}, L_{f_y}, L_g, L_{g_{xy}}, L_{g_{yy}}$-Lipschitz continuous. For any fixed $y$, $\nabla_x f(\cdot, y; \xi)$, $\nabla_y f(\cdot, y; \xi)$, $\nabla^2_{xy}g(\cdot, y;\phi)$, $\nabla^2_{yy}g(\cdot, y;\phi)$ are $\bar L_{f_x}, \bar L_{f_y}, \bar L_{g_{xy}}, \bar L_{g_{yy}}$-Lipschitz continuous.}

\noindent\textbf{Assumption 2 (strong convexity of lower-level objective).}
\emph{For any fixed $x$, $g(x,y)$ is $\mu_g$-strongly convex in $y$, that is, $\nabla^2_{yy}g(x,y)\succeq \mu_g I$.}

Assumptions 1 and 2 together ensure that the first- and second-order derivations of $f(x,y), g(x,y)$ as well as the solution mapping $y^*(x)$ are well-behaved. 

\noindent\textbf{Assumption 3 (stochastic derivatives).}
\emph{The stochastic derivatives $\nabla_x f(x,y;\xi)$, $\nabla_y f(x,y;\xi)$, $\nabla_y g(x,y;\phi)$, $\nabla_{xy}^2g(x, y, \phi)$, and $\nabla_{yy}^2g(x, y, \phi)$ are unbiased estimators of $\nabla_x f(x,y)$, $\nabla_y f(x,y)$, $\nabla_y g(x,y)$, $\nabla_{xy}^2g(x, y)$, and $\nabla_{yy}^2g(x, y)$, respectively; and their variances are bounded by $\sigma_{f_x}^2, \sigma_{f_y}^2$, $\sigma_{g_y}^2$, $\sigma_{g_{xy}}^2, \sigma_{g_{yy}}^2$, respectively. Moreover, their moments are bounded by }
\begin{subequations}
\begin{align}\label{eq.assp_grad_norm}
&\EE_{\xi}[\|\nabla_x f(x,y;\xi)\|^p]\leq C_{f_x}^p,~p=2, 4\nonumber\\
&\EE_{\xi}[\|\nabla_y f(x,y;\xi)\|^p]\leq C_{f_y}^p,~p=2, 4\\
	&\EE_{\phi}[\|\nabla^2_{xy} g(x,y;\phi)\|^2]\leq  C_{g_{xy}}^2,\nonumber\\
	&\EE_{\phi}[\|\nabla^2_{yy} g(x,y;\phi)\|^2]\leq C_{g_{yy}}^2.
\end{align}
\end{subequations}

Assumption 3 is the counterpart of the unbiasedness and bounded variance assumption in the single-level stochastic optimization, which are standard also in \citep{ghadimi2018bilevel,hong2020ac}. In addition, the bounded moments in Assumption 3 ensure the Lipschitz continuity of the upper-level gradient $\nabla F(x)$. 



We first highlight the inherent challenge of directly applying the single-level SGD method \citep{robbins1951} to the bilevel problem \eqref{opt0}. 
To illustrate this point, we derive the gradient of the upper-level function $F(x)$ in the next proposition by analyzing the lower-level optimality condition; see the proof in Appendix C. 
\begin{proposition}\label{prop1}
Under Assumption 2, we have the gradients
\begin{subequations}\label{grad-deter-2}
	\begin{align}
\!\!\!\!&\nabla_xy^*(x)^{\top}\!:=\!-\nabla_{xy}^2g(x, y^*(x))\!\left[\nabla_{yy}^2 g(x, y^*(x))\right]^{-1}\!\!\!\label{grad-deter-2-1}\\ 
\!\!\!\!&\nabla F(x)=\nabla_xf(x, y^*(x))+\nabla_xy^*(x)^{\top}\nabla_y f(x, y^*(x)).\!\label{grad-deter-2-2}
\end{align}
\end{subequations}
\end{proposition}
Note that the gradient $\nabla F(x)$ contains the second-order information of the lower-level problem $g(x, y)$ since it depends on the sensitivity of the lower-level solution $y^*(x)$. The sensitivity of the solution $y^*(x)$ for a strongly-convex program has also been explored in the time-varying convex optimization literature through the lens of perturbation analysis; see e.g., \citep{simonetto2016class}. 
Therefore, we hope that the sample complexity results of bilevel optimization in this paper will also stimulate future research in time-varying convex optimization.

In addition, notice that obtaining an unbiased stochastic estimate of $\nabla F(x)$ and applying SGD on $x$ face two main difficulties: \textbf{(D1)} the gradient $\nabla F(x)$ at $x$ depends on the minimizer of the lower-level problem $y^*(x)$; \textbf{(D2)} even if $y^*(x)$ is known, it is hard to apply the stochastic approximation to obtain an unbiased estimate of $\nabla F(x)$ since $\nabla F(x)$ is nonlinear in $\nabla_{yy}^2 g(x, y^*(x))$; see the discussion of (D2) in stochastic compositional optimization literature, e.g., \citep{wang2017mp,chen2020scsc}. 


Similar to some existing algorithms for bilevel problems, our method addresses (D1) by evaluating $\nabla F(x)$ on a certain vector $y$ in place of $y^*(x)$, but it differs in how to recursively update $y$ and how to address (D2).  
Resembling the definition \eqref{grad-deter-2} with $y^*(x)$ replaced by $y$, we introduce the notation 
\begin{align}\label{grad-deter-3}
\overline{\nabla}_xf\left(x, y\right):=&\nabla_xf\left(x, y\right)-\nabla_{xy}^2g\left(x, y\right)\times\nonumber\\ &\left[\nabla_{yy}^2 g\left(x, y\right)\right]^{-1}\nabla_y f\left(x, y\right). 
\end{align}
As we will show in Lemma \ref{lemma_lip} of Appendix, Assumptions 1-3 ensure that $\nabla F(\cdot)$, $\overline{\nabla}_x f(x,\cdot)$, and $y^*(\cdot)$ are all Lipschitz continuous with constants $L_F, L_f, L_y$, respectively.

\subsection{A single-timescale bilevel method}

Before we present our method, we first review a successful recent effort.
To overcome the difficulty of applying plain-vanilla SGD, a \emph{two-timescale} stochastic approximation (TTSA) algorithm has been recently developed in \citep{hong2020ac}. 
TTSA is a single-loop algorithm and amenable to efficient implementation. 
It consists of two sequences $\{x^k\}$ and $\{y^k\}$: for a given $x^k$, $y^k$ estimates the minimizer $y^*(x^k)$; and, $x^k$ estimates the minimizer $x^*$. 
For notational brevity, we define
\begin{align*}
&h_g^k := \nabla_yg(x^k,y^k;\phi^k),~~~   h_{yy}^k(\phi):=\nabla_{yy}^2 g(x^k,y^k;\phi),\\
&h_{xy}^k(\phi):=\nabla_{xy}^2g(x^k,y^k;\phi).\numberthis
\end{align*}
With $\alpha_k$ and $\beta_k$ denoting two sequences of stepsizes, the TTSA recursion is given by 
\begin{subequations}\label{eq.TTSA}
	\begin{align}
	y^{k+1}&=y^k-\beta_k h_g^k   \label{eq.TTSA-1}\\
	x^{k+1}&=\mathcal{P}_\mathcal{X}\left(x^k-\alpha_k \left(\nabla_x f(x^k, y^k;\xi^k)\right.\right.\nonumber\\
	&\quad\left. \left.-h_{xy}^k(\phi^k)\nabla_{yy}^{-1}\nabla_y f(x^k, y^k;\xi^k)\right)\right) \label{eq.TTSA-2}
\end{align}
\end{subequations}
where $\nabla_{yy}^{-1}$ is a mini-batch approximation of $\left[\nabla_{yy}^2 g(x^k, y^k)\right]^{-1}$. 
The timescale separation refers to the different order of stepsizes used in updating multiple variables. 
To ensure convergence, TTSA requires $y^k$ to be updated in a timescale faster than that of $x^k$ so that $x^k$ is relatively static with respect to $y^k$; i.e., $\lim_{k\rightarrow \infty}\alpha_k/\beta_k=0$ \citep{hong2020ac}. This is termed the two-timescale update.
However, this prevents TTSA from choosing the stepsize ${\cal O}(1/\sqrt{k})$ as SGD, and also results in its \emph{suboptimal complexity}.

\begin{algorithm}[t]
\caption{STABLE for stochastic bilevel problems}\label{alg: STABLE}
    \begin{algorithmic}[1]
    \State{\textbf{initialize:} $x^0, y^0, H_{xy}^0, H_{yy}^0$,  stepsizes $\{\alpha_k, \beta_k\}$.} 
        \For{$k=0,1,\ldots, K-1$}
           \State{compute $h_{xy}^{k-1}(\phi^k)$ and $h_{xy}^k(\phi^k)$}\\~~~\Comment{randomly select datum $\phi^k$}
            \State{update $H_{xy}^k$ via \eqref{eq.STABLE-H1} }
           \State{compute $h_{yy}^{k-1}(\phi^k)$ and $h_{yy}^k(\phi^k)$}
            \State{update $H_{yy}^k$ via \eqref{eq.STABLE-H3} }
           \State{compute $\nabla_xf\left(x^k, y^k;\xi^k\right)$, $\nabla_y f\left(x^k, y^k;\xi^k\right)$}\\~~~\Comment{randomly select datum $\xi^k$}
            \State{update $x^k$ and $y^k$ via \eqref{eq.STABLE}}
        \EndFor
    \end{algorithmic}
  \end{algorithm}

We find that the key reason preventing TTSA from using a single-timescale update is its undesired stochastic upper-level gradient estimator \eqref{eq.TTSA-2} that uses an inaccurate lower-level variable $y^k$ to approximate $y^*(x^k)$. 
With more insights given in Section \ref{subsec.ode}, we propose a new stochastic bilevel optimization method based on a new stochastic bilevel gradient estimator, which we term Single-Timescale stochAstic BiLevEl optimization (\textbf{STABLE}) method.  
Its recursion is given by
\begin{subequations}\label{eq.STABLE}
	\begin{align}
\!\!\!\!    x^{k+1}&=\mathcal{P}_\mathcal{X}\left(x^k \!-\! \alpha_k \left(\nabla_xf(x^k, y^k;\xi^k)\right.\right.\nonumber\\
&\quad-\left.\left.H_{xy}^k(H_{yy}^k)^{-1}\nabla_y f(x^k, y^k;\xi^k)\right) \right)\label{eq.STABLE2}\\
\!\!\!\!   y^{k+1}&=y^k \!-\beta_k h_g^k - (H_{yy}^k)^{-1}(H_{xy}^k)^{\top}(x^{k+1}-x^k).\label{eq.STABLE3}
\end{align}
\end{subequations}
where ${\cal P}_{\cal X}$ denotes the projection on set $\cal X$. In \eqref{eq.STABLE}, the estimates of second-order derivatives are updated as (with stepsize $\tau_k>0$)
\begin{subequations}\label{eq.STABLE-H}
\begin{align}
	\!\! \!	H_{xy}^k&=\overline{\cal P}\left((1-\tau_k)\Big( H_{xy}^{k-1}\!-\!h_{xy}^{k-1}(\phi^k)\Big) \!+\!h_{xy}^k(\phi^k) \right)    \!\!              \label{eq.STABLE-H1}\\
\!\!\! H_{yy}^k&=  \underline{\cal P}\left((1-\tau_k) \Big( H_{yy}^{k-1}\!-\!h_{yy}^{k-1}(\phi^k)\Big)\!+\!h_{yy}^k(\phi^k)\right)	 \!\! \label{eq.STABLE-H3}
\end{align}
\end{subequations}
where $\overline{\cal P}$ is the projection to set $\{X:\|X\|\leq C_{g_{xy}}\}$ and $\underline{\cal P}$ is the projection to set $\{X: X\succeq \mu_g I\}$.

Compared with \eqref{eq.TTSA} and other existing algorithms, the unique features of STABLE lie in: \textbf{(F1)} its $y^k$-update that will be shown to better ``predict'' the next $y^*(x^{k+1})$; and, \textbf{(F2)} a recursive update of $H_{xy}^k, H_{yy}^k$ that is motivated by the advanced variance reduction techniques for single-level nonconvex optimization problems such as STORM \citep{cutkosky2019momentum}, Hybrid SGD\citep{tran2021hybrid} and the recent stochastic compositional optimization method \citep{chen2020scsc}. The marriage of (F1)-(F2) enables STABLE to have a better estimate of $\nabla F(x^k)$, which is responsible for its improved convergence.
Note that we use three stepsizes $\alpha_k$, $\beta_k$ and $\tau_k$ in \eqref{eq.STABLE}, we call our method a single-timescale algorithm because the upper- and lower-level variables use the same order of stepsizes that decrease at the same rate as that of SGD. 
As we will show later, for a class of bilevel problems, the single-timescale recursion \eqref{eq.STABLE} achieves the same convergence rate as SGD for single-level problems. See a summary of STABLE in Algorithm \ref{alg: STABLE}.

\textbf{Remark.}
The projection in \eqref{eq.STABLE-H} is introduced for our current analysis. 
However, projection in \eqref{eq.STABLE-H1} is not uncommon in stochastic algorithms to ensure stability, and the eigenvalue truncation in \eqref{eq.STABLE-H3} is a usual subroutine in Newton-based methods, which is also referred to the positive definite truncation \citep{nocedal2006numerical,paternain2019newton}. 
One potential way to avoid it is to replace \eqref{eq.STABLE-H} with a trust region computation.



\subsection{Continuous-time ODE analysis}\label{subsec.ode} 
Similar to the stochastic compositional optimization \citep{chen2020scsc}, we provide some intuition of our algorithm design via an ODE for the deterministic problem \eqref{opt0-deter}. 
To minimize $F(x)$, we use an ODE analysis to design a continuous dynamic
\begin{equation}\label{eq.xdyn}
    \dot{x}(t) = -\alpha {\cal T}(x(t),y(x(t)))
\end{equation}
by choosing an operator ${\cal T}$.
For single-level minimization of a smooth function $h(x(t))$, one can use the gradient flow $\dot{x}(t) = - \alpha \nabla h(x(t))$. For bilevel minimization \eqref{opt0-deter}, however, we shall avoid ${\cal T}(x,y) = \nabla_x \left(f(x,y^*(x)\right)$ and instead use $y$ to approximate $y^*(x)$. Here note that we have dropped $(t)$ for conciseness.
Hence, define the operator as
\begin{align}\label{eq.ode_x}
   {\cal T}(x,y):=& \nabla_x f(x,y)\!-\!\nabla_{xy}^2g(x, y)[\nabla_{yy}^2 g(x, y)]^{-1}\times\nonumber\\
   &\nabla_yf(x,y) 
 \stackrel{\eqref{grad-deter-3}}{=} \overline{\nabla}_x f(x,y). 
\end{align}

Here, the variable $y$ follows another dynamic that we specify below, which accompanies the $x$-dynamic \eqref{eq.xdyn}.
We will also find a Lyapunov function $V$ such that 
\begin{center}
	{\bf (C1)} $\dot V< 0$; ~~~~~~~~~~~~~~~~~~~~~~~~~~~~~~~~~~~~~~~~~~~~~~~~~~~~\\ 
	
	{\bf (C2)} $\dot V=0$ if and only if $\nabla F(x)=0$ and $y=y^*(x)$.
\end{center} 
If the $\dot{x}$ and $\dot{y}$ dynamics drive an appropriate Lyapunov function $V$ satisfying (C1) and (C2), then $x$ converges to a stationary point of the upper-level problem and $y$ converges to the solution of the lower-level problem. 
 
We first state the results for the continuous-time dynamics below. 
\begin{customthm}{1}[Continuous-time dynamics]\label{prop2}
If we define the $x$- and $y$-dynamics as 
	\begin{align}\label{eq.ode_stable}
    \dot{x}&= -\alpha \nabla_x f(x,y)\!-\alpha\nabla_{xy}^2g(x, y)[\nabla_{yy}^2 g(x, y)]^{-1}\nabla_yf(x,y)\nonumber\\
    \dot{y}& = -\beta\nabla_yg(x,y) - \left[\nabla_{yy}^2 g\left(x, y\right)\right]^{-1}\nabla_{yx}^2g\left(x, y\right)\dot{x}
\end{align}
and choose the constants $\alpha$ and $\beta$ appropriately, then there exists a Lyapunov function $V$ of the $x$- and $y$-dynamics that satisfies (C1) and (C2).
\end{customthm}

\begin{proof}
To highlight the intuition, we provide a constructive proof of this theorem. 
We first try $V_0:= f(x,y^*(x))$. To clarify, we can use $y^*(x)$ in a Lyapunov function but not in a dynamic to evolve a quantity. In this case, we have
\begin{align*}
   \dot{V}_0&=\dotp{\nabla_xf(x,y^*(x)), \dot{x}
    } + \dotp{\nabla_yf(x,y^*(x)), \nabla_xy^*(x)\dot{x}} \\
    &=\dotp{\nabla_x f(x,y^*(x))+\nabla_xy^*(x)^\top\nabla_yf(x,y^*(x)),\dot{x}}.
\end{align*}

\begin{figure}[t]
    \centering
    \includegraphics[width=.42\textwidth]{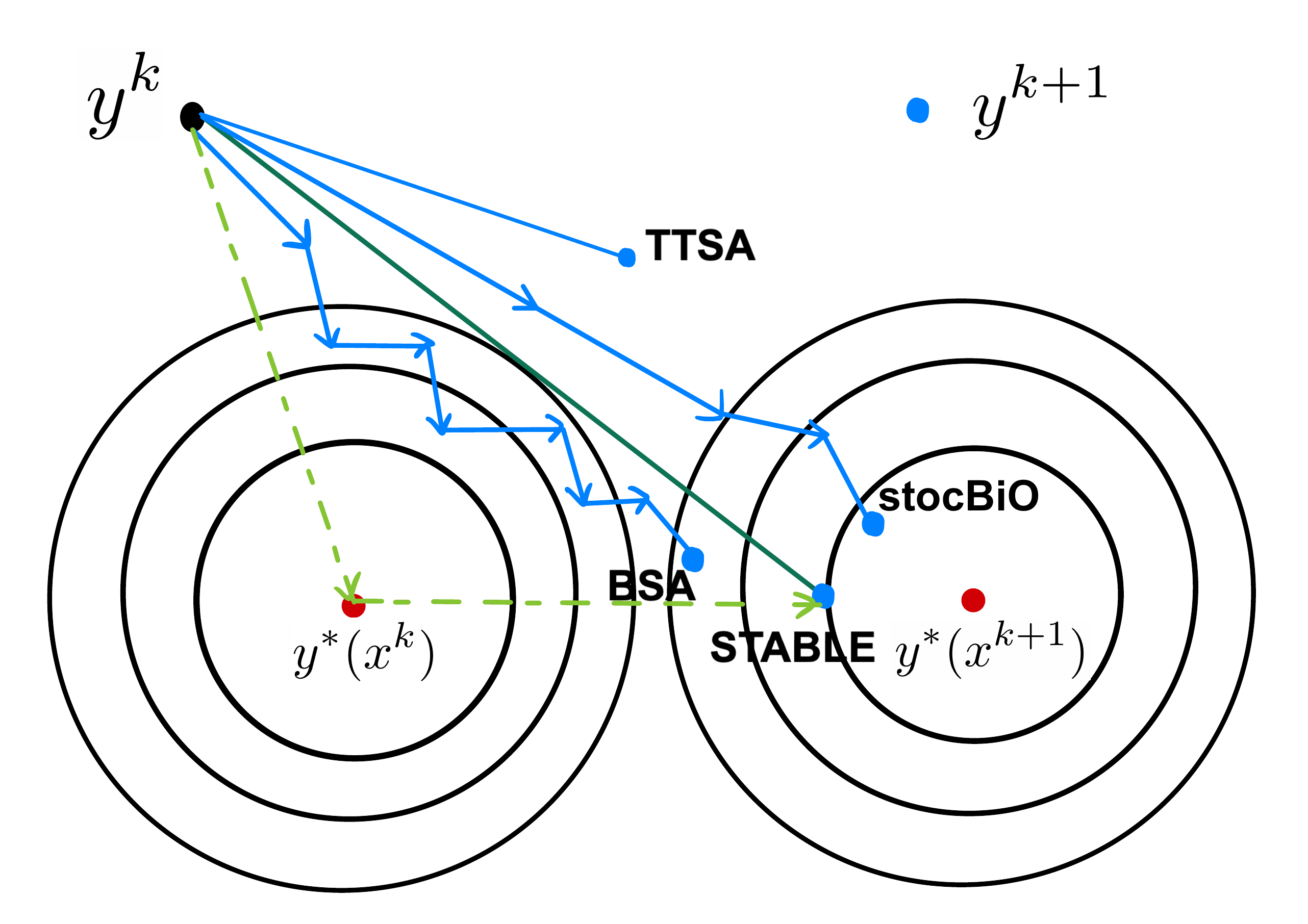}
    \caption{A geometric illustration of the $y^k$ update under the state-of-the-art algorithms; black dot represents $y^k$, red dots represent the lower-level solution $y^*(x^k)$ and $y^*(x^{k+1})$, blue dots represent $y^{k+1}$ under different algorithms, and blue arrow denotes the inner loop updates. \textbf{STABLE} updates $y^k$ by linearly combining the stochastic gradient direction towards $y^*(x^k)$ and the moving direction from $y^*(x^k)$ to $y^*(x^{k+1})$. 
 In contrast, \textbf{BSA} \citep{ghadimi2018bilevel} runs multiple stochastic gradient steps; \textbf{TTSA} \citep{hong2020ac} runs one stochastic gradient step with a smaller stepsize; \textbf{stocBiO} \citep{ji2020provably} runs multiple stochastic gradient steps with an increasing batch size.
    }\label{fig:intuition}
\end{figure}

Recall the definition in \eqref{grad-deter-2}. Then we have
\begin{align*}\label{eq.dot-lyap0}
   \dot{V}_0& = -\alpha \dotp{{\cal T}(x,y^*(x)),{\cal T}(x,y)}\\
   & \stackrel{(a)}{\leq} -\alpha \|{\cal T}(x,y^*(x))\|^2 \nonumber\\
   &\quad+ \alpha \|\overline{\nabla}_x f(x,y) - \overline{\nabla}_x f(x,y^*(x))\|\|{\cal T}(x,y^*(x))\|\\
   & \stackrel{(b)}{\leq} -\alpha \|{\cal T}(x,y^*(x))\|^2\nonumber\\
   &\quad+ \alpha L_f\|y - y^*(x)\|\|{\cal T}(x,y^*(x))\|\\
   &\stackrel{(c)}{\leq} -\frac{\alpha}{2} \|{\cal T}(x,y^*(x))\|^2 + \frac{\alpha L_f^2}{2}\|y - y^*(x)\|^2
    \numberthis
\end{align*}
where (a) uses the Cauchy-Schwarz inequality, (b) follows from the $L_f$-Lipschitz continuity of $\overline{\nabla}_x f(x,\cdot)$ established in Lemma \ref{lemma_lip}, and (c) is due to the Young's inequality.  

To satisfy (C1), we have $\dot{V}_0\le 0$ only if $L_f\|y - y^*(x)\|\le \|{\cal T}(x,y^*(x))\|$, thus, requiring the information of $\|y - y^*(x)\|$ --- not doable without knowing $y^*(x)$.

Let us try to mitigate the term $\|y(x) - y^*(x)\|^2$ by defining the following new Lyapunov function:
\begin{align*}
  V:&= f(x,y^*(x)) + \frac{1}{2}\|y-y^*(x)\|^2  \numberthis \label{eq.cont-lyap}
\end{align*}
which implies that
\small
\begin{align*}
\!\!\!\dot{V}&\!=\!-\alpha \dotp{{\cal T}(x,y^*(x)),{\cal T}(x,y)}  \! +\! \dotp{y-y^*(x), \dot{y} - \nabla_xy^*(x)\dot{x}}\\
    &\stackrel{\eqref{eq.dot-lyap0}}{\leq} -\frac{\alpha}{2} \|{\cal T}(x,y^*(x))\|^2 + \frac{\alpha L_f^2}{2}\|y - y^*(x)\|^2\nonumber\\
    &~~~  + \dotp{y-y^*(x), \dot{y} - \nabla_xy^*(x)\dot{x}}\numberthis \label{eq.ode-V3-2}\\
 &\leq -\frac{\alpha}{2} \|{\cal T}(x,y^*(x))\|^2 -\left(\beta- \frac{\alpha L_f^2}{2}\right)\|y - y^*(x)\|^2\\
    &\quad+ \dotp{y-y^*(x), \dot{y}+\beta(y - y^*(x)) - \nabla_xy^*(x)\dot{x}}\! \! \numberthis\label{eq.ode-V3}
\end{align*}
\normalsize
where $\beta>0$ is a fixed constant.
The first two terms in the RHS of \eqref{eq.ode-V3} are non-positive given that $\alpha\geq 0$ and $\beta\geq {\alpha L_f^2}/{2}$, but the last term can be either positive or negative. To control the last term and thus ensure the descent of $V(t)$, we are motivated to use a $y$-dynamic like
\begin{equation}\label{eq.ode-y0}
    \dot{y} 
    \approx -\beta(y - y^*(x)) + \nabla_xy^*(x)\dot{x}.
\end{equation}
To avoid using $y^*$ in a dynamic, we approximate $y - y^*(x)$ by $\nabla_yg(x,y)$ and $\nabla_xy^*(x)$ by (cf. \eqref{grad-deter-2-1})
\begin{equation}\label{eq.ode-y1}
    \nabla_xy(x):= - \left[\nabla_{yy}^2 g\left(x, y\right)\right]^{-1}\nabla_{xy}^2g\left(x, y\right).
\end{equation}
These choices lead to the $y$-dynamics:
\begin{equation}\label{eq.ode-y2}
    \dot{y} = -\beta\nabla_yg(x,y) + \nabla_xy(x)\dot{x}.
\end{equation}

Although we approximate \eqref{eq.ode-y0} by \eqref{eq.ode-y2}, we will plug $y$-dynamics \eqref{eq.ode-y2} into \eqref{eq.ode-V3} and show that $V$ satisfies (C1). 
Specifically, plugging \eqref{eq.ode-y2} into \eqref{eq.ode-V3-2} leads to
\begin{align*}\label{eq.ode-inner}
	 &\dotp{y-y^*(x), \dot{y} - \nabla_xy^*(x)\dot{x}} \nonumber\\
	=& -\dotp{y-y^*(x),  \beta\nabla_yg(x,y) \\
	&- \nabla_xy(x)\dot{x} +\nabla_xy^*(x)\dot{x}}.\numberthis
\end{align*}
As $g(x,\cdot)$ is $\mu_g$-strongly convex by Assumption 2, we have
\small
\begin{equation}\label{eq.ode-ysc}
	 \big\langle y-y^*(x), \nabla_y g(x,y) -\nabla_y g(x,y^*(x))\big\rangle\geq \mu_g\|y-y^*(x)\|^2
\end{equation}
\normalsize
where $\nabla_y g(x,y^*(x)) = 0$ as $y^*(x)$ minimizes $g(x,y)$.

Therefore, plugging \eqref{eq.ode-ysc} into \eqref{eq.ode-inner}, we have
\vspace{-0.1cm}
\begin{align*}\label{eq.ode-inner2}
 &\dotp{y-y^*(x), \dot{y} - \nabla_xy^*(x)\dot{x}}\\
\le &  \!-\!\dotp{y-y^*(x),  ( \nabla_xy^*(x)\!-\! \nabla_xy(x))\dot{x}}\!-\!\beta \mu_g\|y-y^*(x)\|^2\\
\le & \|y-y^*(x)\|\|\nabla _xy^*(x) - \nabla_xy(x)\|\|\dot{x}\|-\!\beta \mu_g\|y-y^*(x)\|^2\\
\le & \alpha B_x L_y\|y-y^*(x)\|^2-\!\beta \mu_g\|y-y^*(x)\|^2\numberthis
\end{align*}

where the second inequality uses the Cauchy-Schwarz inequality, and the last inequality follows the bound $B_x$ of $\|\dot{x}\|$ and the Lipschitz constant $L_y$ of $\nabla_xy(x)$, both of which can be derived from Assumptions 1--3. 

Now plugging \eqref{eq.ode-inner2} into \eqref{eq.ode-V3-2}, we have
\begin{align*}
    \dot{V}   \leq &-\frac{\alpha}{2}\|{\cal T}(x,y^*(x))\|^2 \nonumber\\
    &- \left(\beta\mu_g-\frac{\alpha L_f^2}{2} - \alpha B_x L_y\right)\|y-y^*(x)\|^2.\numberthis
\end{align*}

Now let us check (C1) and (C2). To ensure $\dot{V}\leq  0$ in (C1), we can set $\alpha\leq\frac{2\mu_g\beta}{L_f^2 + 2B_xL_y}$. For (C2), we have $\dot{V}=0$ if and only if $y=y^*(x)$ and ${\cal T}(x,y^*(x))=\nabla F(x)=0$. 
\end{proof}

 With the insights gained from the continuous-time update \eqref{eq.ode_stable}, 
our stochastic update \eqref{eq.STABLE} essentially discretizes time $t$ into iteration $k$, and replaces the first- and second-order derivatives in $\dot x$ and $\dot y$ by their recursive (variance-reduced) stochastic values in \eqref{eq.STABLE-H}.  

 
\textbf{Remark.}
	The key ingredient of our STABLE method is the design of the lower-level update on $y^k$, which leads to a more accurate stochastic estimate of $\nabla F(x^k)$. See a comparison of the $y$-update with other algorithms in Figure \ref{fig:intuition}.
	In the update \eqref{eq.STABLE}, we implement the SGD-like update for the upper-level variable $x^k$. With the lower-level $y^k$ update unchanged, it is easy to apply SGD-improvement techniques such as momentum and variance reduction, to accelerate the convergence of STABLE. This will help STABLE achieve state-of-the-art performance for stochastic bilevel optimization.

\section{Convergence Analysis}
\label{sec.ic-ana}
In this section, we establish the convergence rate of our single-timescale STABLE algorithm. We will highlight the key steps of the proof and leave the detailed analysis in Appendix.

\paragraph{Moreau Envelop.}
Different from problem \eqref{opt0-deter}, \eqref{opt0} tackles the constraint on a convex and closed set $\mathcal{X}$. To levarage the ODE analysis to the constraint case, for fixed $\rho>0$, we define the Moreau envelop and proximal map as follows.
\begin{align}
    &\Phi_{1 / \rho}(z):=\min _{x \in \mathcal{X}}\left\{F(x)+(\rho / 2)\|x-z\|^{2}\right\}\nonumber\\ 
    &\widehat{x}(z):=\underset{x \in  \mathcal{X}}{\arg\min}\left\{F(x)+(\rho / 2)\|x-z\|^{2}\right\}
\end{align}
For any $\epsilon>0$, we use the definition in \citep{davis2018stochastic} that $x^k\in\mathcal{X}$ is an $\epsilon$-nearly stationary solution if $x^k$ satisfies the following condition 
\begin{align}\label{near-stationary}
    \EE\left[\|\widehat{x}(x^k)-x^k\|^2\right]\leq \rho^2\epsilon.
\end{align}

In Section \ref{convergence_main}, we will utilize the near-stationarity condition \eqref{near-stationary} as a tool to quantify the convergence of STABLE when $F(x)$ is non-convex. 

\subsection{Main results}\label{convergence_main}
We first present the result of our algorithm when the upper-level function $F(x)$ is nonconvex in $x$. We need the following additional assumption.

\vspace{0.2cm}
\noindent\textbf{Assumption 4 (weak convexity).}
\emph{Function $F(x)$ is $\mu_F$-weakly convex in $x$, that is, $\nabla^2_{xx}F(x)\succeq \mu_F I$.} Note that $\mu_F$ is not necessarily positive. 

\vspace{0.4cm}

For simplicity of the convergence analysis, we define the following Lyapunov function 
\begin{align*}\label{eq.Lya}
    &\mathbb{V}^k \!:=\! \Phi_{1/\rho}(x^k) +  \|y^k - y^*(x^k)\|^2 \\
    &~~~~~~~+\|H_{yy}^k - \nabla_{yy}^2g(x^k, y^k)\|^2\\
    &~~~~~~~+  \|H_{xy}^k - \nabla_{xy}^2g(x^k, y^k)\|^2 \numberthis
\end{align*}
which mimics the continuous-time Lyapunov function \eqref{eq.cont-lyap} for the deterministic problem. Similar to the ODE analysis, we need to quantify the difference between two Lyapunov functions $\mathbb{V}^{k+1}-\mathbb{V}^k$. We will first analyze the descent of the Moreau Envelop of the upper-level objective in the next lemma. 
\begin{lemma}[Descent of the upper level]
\label{lemma3}
Under Assumptions 1--4, the sequence of $x^k$ satisfies 
\begin{align*}\label{eq.lemma3}
&\quad\EE[\Phi_{1/\rho}(x^{k+1})]-  \EE[\Phi_{1/\rho}(x^k)]\\ 
&\leq -\frac{(\mu_F+\rho)\rho\alpha_k}{4}\|\widehat{x}(x^k)-x^k\|^2+2 \rho\alpha_k^2\left(C_{f_x}^2 + \frac{C_{g_{xy}^2}C_{f_y}^2}{\mu_g^2}\right)\\
&~~~ +\frac{2L_f^2\rho \alpha_k}{\mu_F+\rho}\|y^k-y^*(x^k)\|^2\\
&~~~ +\frac{4C_{f_y}^2C_{g_{xy}}^2\rho \alpha_k}{(\mu_F+\rho)\mu_g^4} \EE[\|H_{yy}^k - \nabla_{yy}^2(x^k,y^k)\|^2]\\
&~~~+ \frac{4C_{f_y}^2\rho \alpha_k}{(\mu_F+\rho)\mu_g^2}  \EE[\|H_{xy}^k-\nabla_{xy}^2(x^k,y^k)\|^2]\numberthis
\end{align*}
where $L_f, L_F$ are defined in Lemma \ref{lemma_lip} of Appendix, and $C_{g_{xy}}$ is the projection radius in \eqref{eq.STABLE-H1}. 
\end{lemma}

Lemma \ref{lemma3} implies that the descent of the Moreau Envelop of the upper-level objective functions depends on the error of the lower-level variable $y^k$, and the estimation errors of $H_{yy}^k$ and $H_{xy}^k$. After bounding all of them in Lemmas~\ref{lemma2} and \ref{scsc-lemma2} in Appendix, we can get the following convergence result. 

\begin{customthm}{2}[Nonconvex]\label{theorem1}
Under Assumptions 1--4 and setting $\rho>|\mu_F|$, if we choose the stepsizes as 
\small
\begin{subequations}\label{eq.stepsize}
\begin{align}
\!\!&\beta_k\leq \min\Bigg\{\frac{1}{\sqrt{K}}, \frac{\mu_g /L_g}{32(\mu_g+L_g) c}\Bigg\}\label{eq.stepsize_1}\\
\!\!&\alpha_k\leq\min\Bigg\{\beta_k,\frac{(c+4\rho C_{g_{xy}}^2C_{f_y}^2/(\mu_F+\rho)\mu_g^4)^{-1}}{\sqrt{K}}, \nonumber\\ 
\!\!&\frac{(c+4\rho C_{f_y}^2/(\mu_F+\rho)\mu_g^2)^{-1}}{\sqrt{K}},	~\frac{\mu_g L_g\beta_k/(\mu_g+L_g)}{2(c+2\rho L_f^2/(\mu_F+\rho))} \Bigg\}\label{eq.stepsize_2}
\end{align}
\normalsize
\end{subequations}
\normalsize
and $\tau_k=\frac{1}{\sqrt{K}}$, then the iterates $\{x^k\}$ and $\{y^k\}$ satisfy
	\begin{align}\label{eq.theorem1}
    &\frac{1}{K}\sum_{k=1}^K\EE\left[\left\|\widehat{x}(x^k)-x^k\right\|^2\right] = {\cal O}\left(\frac{1}{\sqrt{K}}\right)~~{\rm and}\nonumber\\
    &\EE\left[\left\|y^K-y^*(x^K)\right\|^2\right] = {\cal O}\left(\frac{1}{\sqrt{K}}\right)
\end{align}
where $y^*(x^K)$ is the minimizer of the problem  \eqref{opt0-low}, and $c>0$ is a constant that is independent of the stepsizes $\alpha_k, \beta_k, \tau_k$ and the number of iterations $K$.
\end{customthm}

Theorem \ref{theorem1} implies that the convergence rate of STABLE to the stationary point of \eqref{opt0} is ${\cal O}(K^{-\frac{1}{2}})$.
Since each iteration of STABLE only uses two samples (see Algorithm \ref{alg: STABLE}), the sample complexity to achieve an $\epsilon$-stationary point of \eqref{opt0} is ${\cal O}(\epsilon^{-2})$, which is on the same order of SGD's sample complexity for the single-level nonconvex problems \citep{ghadimi2013sgd}, and significantly improves the state-of-the-art single-loop TTSA's convergence rate ${\cal O}(\epsilon^{-2.5})$ \citep{hong2020ac}. 
In addition, this convergence rate is not directly comparable to other recently developed bilevel optimization methods, e.g., \citep{ghadimi2018bilevel,ji2020provably} since STABLE does not need the increasing batchsize nor double-loop. Regarding the sample complexity, however, STABLE improves over \citep{ghadimi2018bilevel,ji2020provably} by at least the order of ${\cal O}(\log(\epsilon^{-1}))$.

\begin{figure*}[t]
        \vspace*{-0.2cm}
  
  \includegraphics[width=0.33\textwidth]{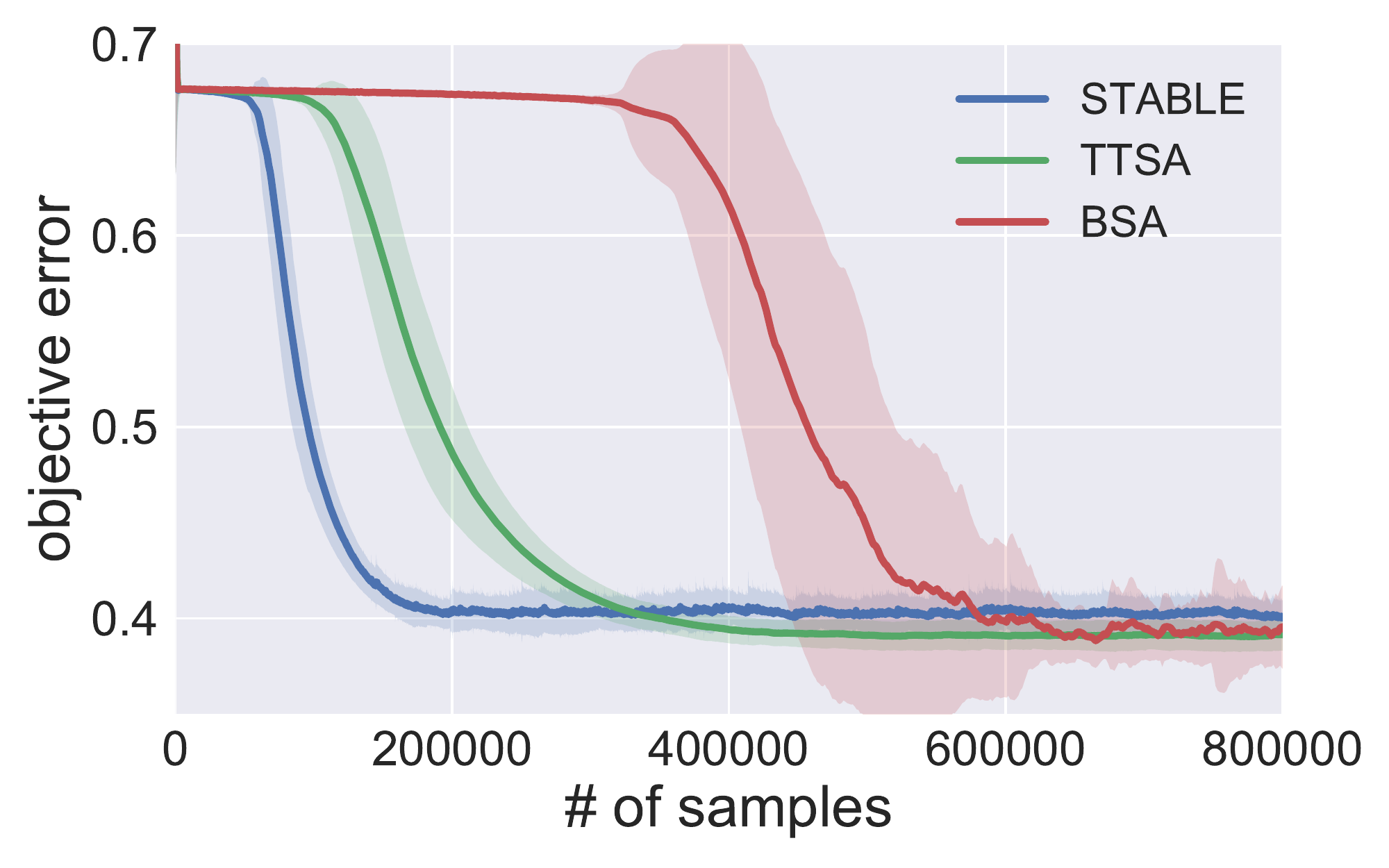}\hspace{-0.1cm}\includegraphics[width=0.33\textwidth]{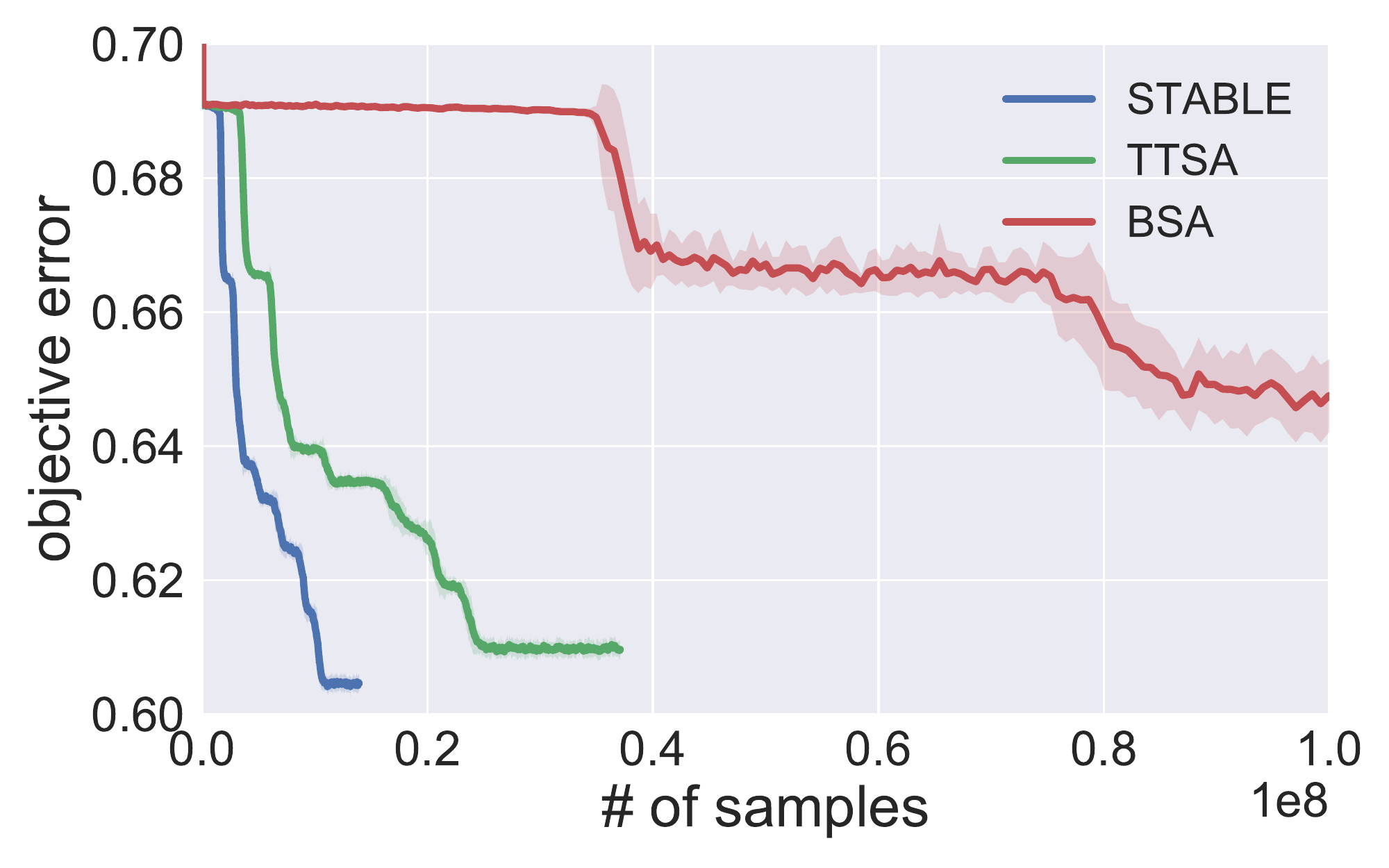}\includegraphics[width=0.33\textwidth]{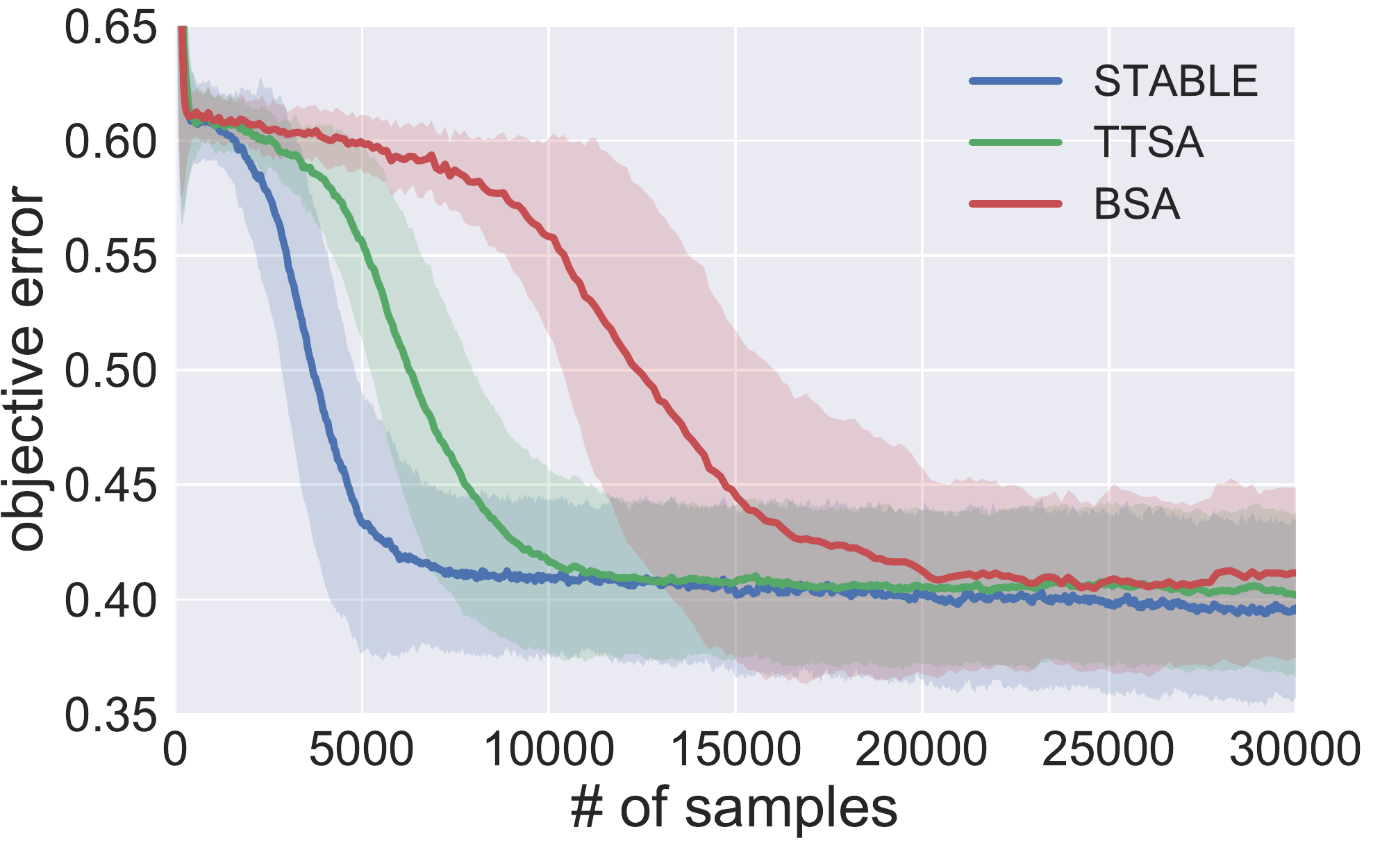}
  \vspace{0.1cm}
  
  \includegraphics[width=0.33\textwidth]{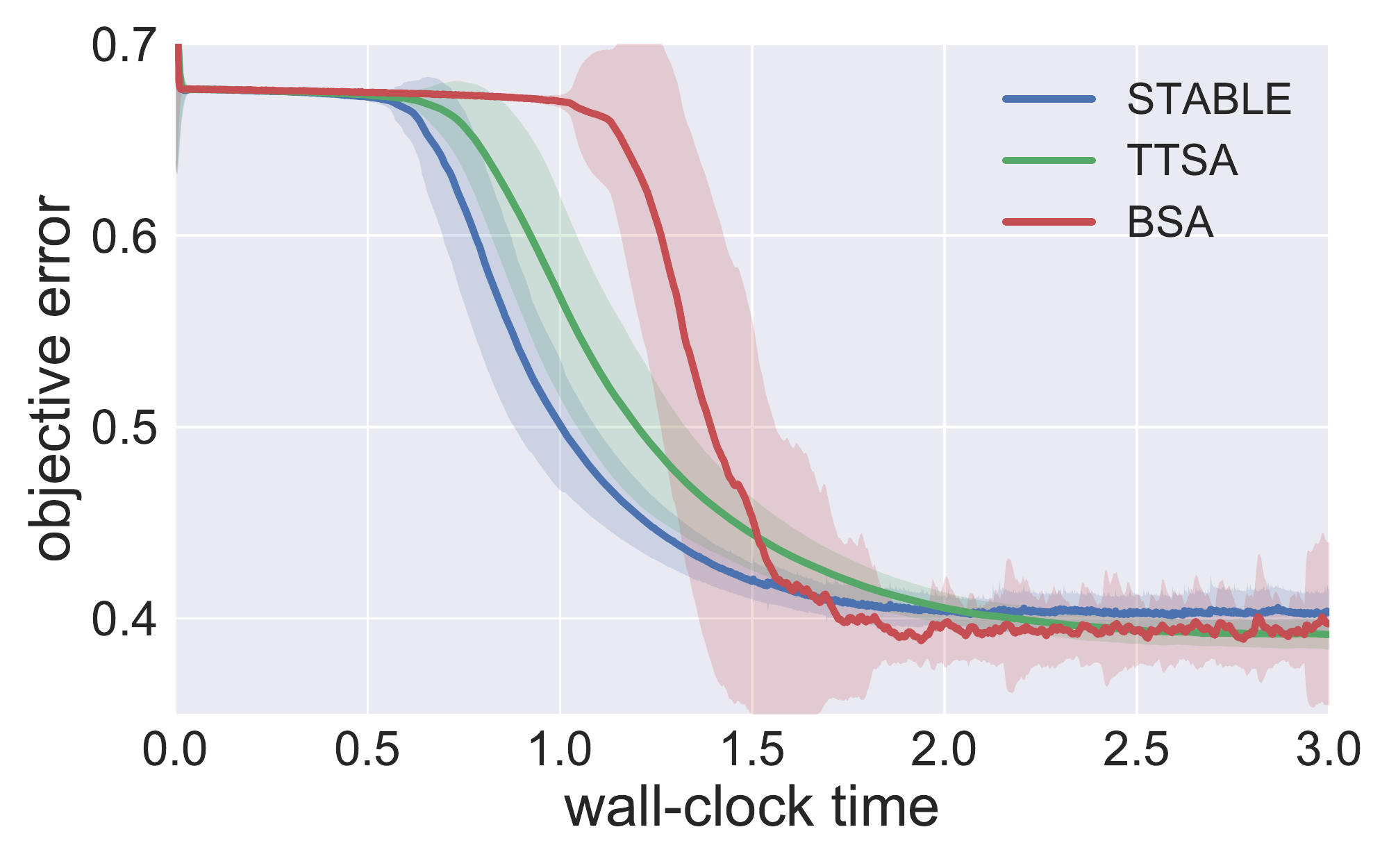}\hspace{-0.1cm}\includegraphics[width=0.33\textwidth]{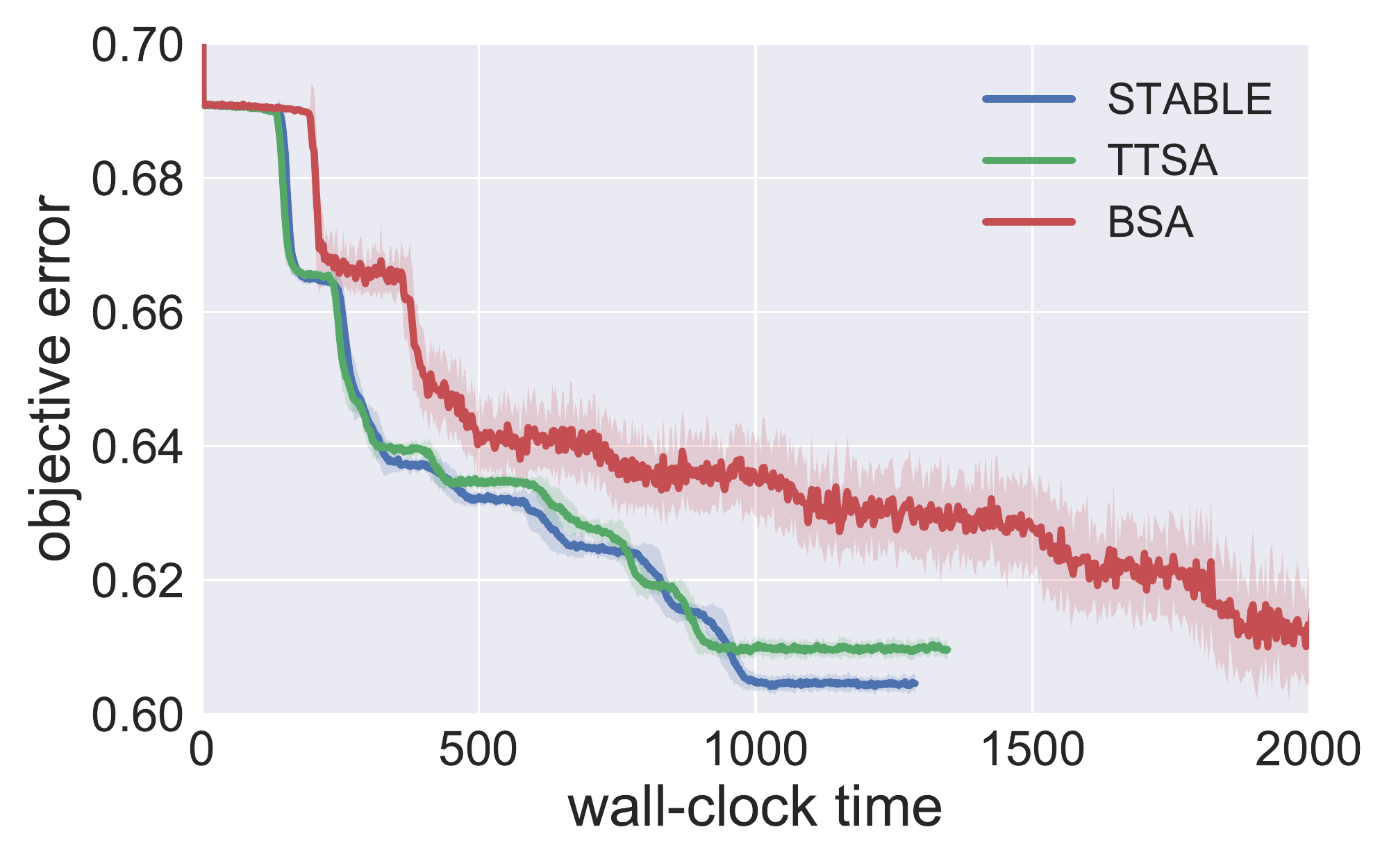}\includegraphics[width=0.33\textwidth]{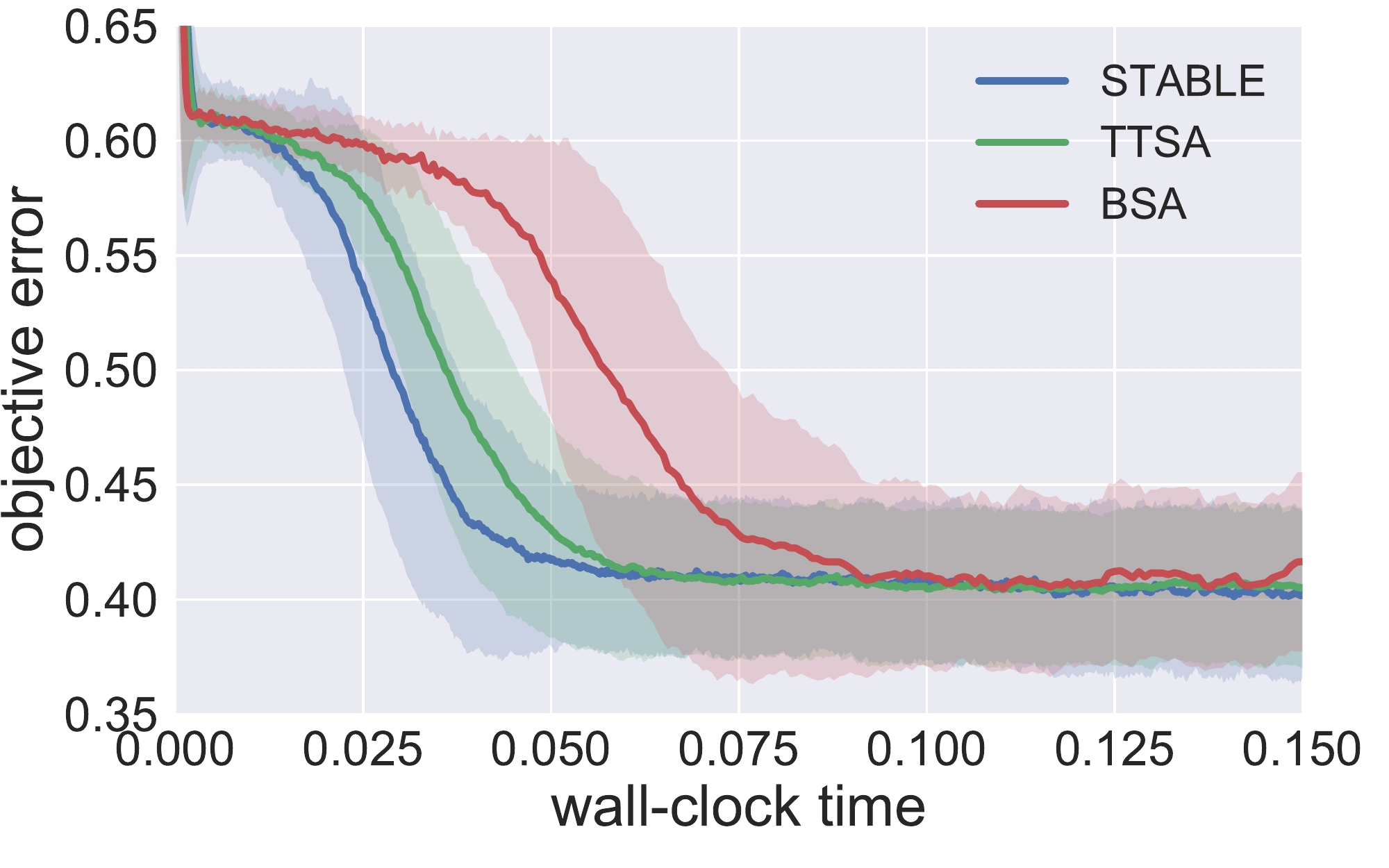}
 \vspace{-0.3cm}
    \caption{The hyper-parameter optimization task on ijcnn1, covtype and australian datasets. The solid line shows the results averaged over $50$ independent trials with random initialization, and the shaded region denotes the standard deviation of results over random trials.}
    \label{fig:hyperpara}
\end{figure*}

We next present the result in the strongly convex case for completeness, where the following additional assumption is needed.
\vspace{0.1cm}

\noindent\textbf{Assumption 5 (strong convexity).}
\emph{Function $F(x)$ is $\mu$-strongly convex in $x$, that is, $\nabla^2_{xx}F(x)\succeq \mu I$.}

Notice that Assumption 5 does not contradict with Assumption 3 since for the constrained upper-level problem \eqref{opt0}, only in the constraint set $\mathcal{X}$ do the gradients need to be bounded. Regarding applications, hyperparameter optimization for linear regression or SVM satisfies this assumption. 

 \begin{customthm}{3}[Strongly convex]\label{theorem2}
Under Assumptions 1--3, 5, if we choose the stepsizes as 
\begin{subequations}\label{eq.stepsize2}
\begin{align}
	&\beta_k=\tau_k\leq \min\left\{\frac{\mu_g/L_g}{32(\mu_g+L_g)}, \frac{1}{K_0+k}\right\}\\
	&\alpha_k\leq\min\left\{\sqrt{\frac{\mu_gL_g}{4c(\mu_g+L_g)}}, \frac{\mu\mu_gL_g}{2L_f^2(\mu_g+L_g)},\right.\nonumber\\ &\left.~~~~~~~~~~~~~~~\frac{1}{\sqrt{4c}},\frac{\mu\mu_g^4}{8C_{g_{xy}}^2C_{f_y}^2}, \frac{\mu\mu_g^2}{8C_{f_y}^2}\right\}\beta_k
\end{align}
\end{subequations}
where $K_0>0$ is a sufficiently large constant and $c>0$ is an absolute constant that is independent of $\alpha_k, \beta_k, \tau_k$, then the iterates $\{x^k\}$ and $\{y^k\}$ satisfy
	\begin{align}\label{eq.theorem2}
    &\EE\left[\left\|x^k-x^*\right\|^2\right] = {\cal O}\left(\frac{1}{k}\right)~~~{\rm and}~~~\nonumber\\
    &\EE\left[\left\|y^k-y^*(x^k)\right\|^2\right] = {\cal O}\left(\frac{1}{k}\right)
\end{align}
where  the solution $x^*$ is defined as $x^*=\argmin_{x \in \cal X} F(x)$ and $y^*(x^k)$ is the minimizer of the lower-level problem in \eqref{opt0-low}.
\end{customthm}
Theorem \ref{theorem2} implies that to achieve an $\epsilon$-optimal solution for both the lower-level and upper-level problems, the sample complexity of STABLE is ${\cal O}(\epsilon^{-1})$. This complexity is on the same order of SGD's complexity for the single-level strongly convex problems \citep{ghadimi2013sgd}, and improves the state-of-the-art single-loop TTSA's sample complexity ${\cal O}(\epsilon^{-2})$ for an $\epsilon$-optimal upper-level solution and ${\cal O}(\epsilon^{-1.5})$ for an $\epsilon$-optimal lower-level solution \citep{hong2020ac}. 
Compared with double-loop bilevel algorithms in this strong-convex case, STABLE also improves over the BSA's query complexity ${\cal O}( \epsilon^{-1} )$ in terms of the stochastic upper-level function and ${\cal O}(\epsilon^{-2})$  in terms of the stochastic lower-level function \citep{ghadimi2018bilevel}.

\section{Numerical Tests}\label{num}

This section evaluates the empirical performance of our STABLE. 
For all compared algorithms, we follow the order of stepsizes suggested in the original papers, and the stepsizes are chosen from $\{1,0.5,0.1,0.05,0.01\}$, e.g., the best one for each algorithm. In our numerical experiments, we compared our method with several state-of-the-art algorithms such as BSA in \citep{ghadimi2018bilevel}, and TTSA in \citep{hong2020ac}. We did not include other recent algorithms such as \citep{ji2020provably,guo2021randomized,khanduri2021momentum} which either require increasing the batch size or adding the acceleration of $x$-update. All the algorithms are implemented using Python 3.6 and run on the same laptop.  
 
 We test all the algorithms in a hyper-parameter optimization task which aims to find the optimal hyper-parameter $x\in \mathbb{R}^d$ (e.g., regularization coefficient), which is used in training a model $y\in\mathbb{R}^d$ on the training set, such that the learned model achieves the low risk on the validation set. 
Let $\ell(y;\xi)$ denote the logistic loss of the model $y$ on datum $\xi$, and ${\cal D_{\rm{val}}}$ and ${\cal D_{\rm{tra}}}$ denote, respectively, the training and validation datasets. 
Specifically, we aim to solve 
\begin{align}\label{eq.hyer-opt}
 &\min_{x\in \mathbb{R}^d}~\EE_{\xi\sim{\cal D_{\rm{val}}}}[\ell(y^*(x);\xi)]\nonumber\\
 &\text{s.t.} ~y^*(x)\in\argmin_{y\in\mathbb{R}^d}\EE_{\phi\sim{\cal D_{\rm{tra}}}}[\ell(y;\phi)] + \sum\limits_{i=1}^dx_iy_i^2. 
\end{align}
In Figure \ref{fig:hyperpara}, we compare the performance of three algorithms on ijcnn1, covtype and australian datasets \citep{CC01a} and report their objective errors versus number of samples and the walk-clock time. 
In all tested datasets, STABLE has sizeable gain in terms of sample complexity compared with the double-loop  or two-timescale algorithms since it uses single-loop and single-timescale update. In addition, although TTSA has more efficient $y$-update, STABLE enjoys the better overall wall-clock time in our simulated setting. This suggests that our STABLE algorithm is preferable in the regime where the sampling is more costly than computation or the dimension $d$ is relatively small, for example in hyperparameter optimization in quantitative trading. 

\section{Conclusions}
\label{sec.concl}
This paper develops a new stochastic gradient estimator for bilevel optimization problems.
When running SGD on top of this stochastic bilevel gradient, the resultant STABLE algorithm runs in a single loop fashion, and uses a single-timescale update. 
In both the nonconvex and strongly-convex cases, STABLE matches the sample complexity of SGD for single-level stochastic problems. 
One possible extension is to apply SGD-improvement techniques to accelerate STABLE, which helps STABLE achieve state-of-the-art performance for bilevel problems. 
Another natural extension is to apply our bilevel optimization method to the general two-timescale stochastic approximation case, in a similar fashion to \citep{dalal2018finite,kaledin2020finite}. Improving the sample complexity of such general case can be of great interest to the reinforcement learning community. 
 
 \section*{Acknowledgment}
The work of T. Chen and Q. Xiao was partially supported by and the Rensselaer-IBM AI Research Collaboration (\url{http://airc.rpi.edu}), part of the IBM AI Horizons Network (\url{http://ibm.biz/AIHorizons}) and NSF 2134168. 

\bibliography{myabrv,bilevelopt}

\begin{thebibliography}{53}
\providecommand{\natexlab}[1]{#1}
\providecommand{\url}[1]{\texttt{#1}}
\expandafter\ifx\csname urlstyle\endcsname\relax
  \providecommand{\doi}[1]{doi: #1}\else
  \providecommand{\doi}{doi: \begingroup \urlstyle{rm}\Url}\fi

\bibitem[Borsos et~al.(2020)Borsos, Mutn{\`y}, and Krause]{borsos2020coresets}
Zal{\'a}n Borsos, Mojm{\'\i}r Mutn{\`y}, and Andreas Krause.
\newblock Coresets via bilevel optimization for continual learning and
  streaming.
\newblock In \emph{Proc. Advances in Neural Info. Process. Syst.}, Virtual,
  December 2020.

\bibitem[Bracken and McGill(1973)]{bracken1973mathematical}
Jerome Bracken and James~T McGill.
\newblock Mathematical programs with optimization problems in the constraints.
\newblock \emph{Operations Research}, 21\penalty0 (1):\penalty0 37--44, 1973.

\bibitem[Chang and Lin(2011)]{CC01a}
Chih-Chung Chang and Chih-Jen Lin.
\newblock {LIBSVM}: A library for support vector machines.
\newblock \emph{ACM Transactions on Intelligent Systems and Technology},
  2:\penalty0 27:1--27:27, 2011.
\newblock Software available at \url{http://www.csie.ntu.edu.tw/~cjlin/libsvm}.

\bibitem[Chen et~al.(2021{\natexlab{a}})Chen, Sun, and Yin]{chen2020scsc}
Tianyi Chen, Yuejiao Sun, and Wotao Yin.
\newblock Solving stochastic compositional optimization is nearly as easy as
  solving stochastic optimization.
\newblock \emph{IEEE Trans. Sig. Proc.}, 69:\penalty0 4937 -- 4948, June
  2021{\natexlab{a}}.

\bibitem[Chen et~al.(2021{\natexlab{b}})Chen, Sun, and Yin]{chen2021closing}
Tianyi Chen, Yuejiao Sun, and Wotao Yin.
\newblock Closing the gap: Tighter analysis of alternating stochastic gradient
  methods for bilevel problems.
\newblock \emph{Advances in Neural Information Processing Systems}, 34,
  2021{\natexlab{b}}.

\bibitem[Colson et~al.(2007)Colson, Marcotte, and Savard]{colson2007overview}
Beno{\^\i}t Colson, Patrice Marcotte, and Gilles Savard.
\newblock An overview of bilevel optimization.
\newblock \emph{Annals of operations research}, 153\penalty0 (1):\penalty0
  235--256, 2007.

\bibitem[Cutkosky and Orabona(2019)]{cutkosky2019momentum}
Ashok Cutkosky and Francesco Orabona.
\newblock Momentum-based variance reduction in non-convex sgd.
\newblock \emph{Proc. Advances in Neural Info. Process. Syst.}, 32, December
  2019.

\bibitem[Dalal et~al.(2018)Dalal, Thoppe, Sz{\"o}r{\'e}nyi, and
  Mannor]{dalal2018finite}
Gal Dalal, Gugan Thoppe, Bal{\'a}zs Sz{\"o}r{\'e}nyi, and Shie Mannor.
\newblock Finite sample analysis of two-timescale stochastic approximation with
  applications to reinforcement learning.
\newblock In \emph{Conference On Learning Theory}, pages 1199--1233, Graz,
  Austria, July 2018.

\bibitem[Daskalakis and Panageas(2018)]{daskalakis2018limit}
Constantinos Daskalakis and Ioannis Panageas.
\newblock The limit points of (optimistic) gradient descent in min-max
  optimization.
\newblock In \emph{Proc. Advances in Neural Info. Process. Syst.}, pages
  9256--9266, Montreal, Canada, December 2018.

\bibitem[Davis and Drusvyatskiy(2018)]{davis2018stochastic}
Damek Davis and Dmitriy Drusvyatskiy.
\newblock Stochastic subgradient method converges at the rate $\mathcal{O}
  (k^{-1/4})$ on weakly convex functions.
\newblock \emph{arXiv preprint arXiv:1802.02988}, 2018.

\bibitem[Dempe and Zemkoho(2020)]{dempe2020bilevel}
Stephan Dempe and Alain Zemkoho.
\newblock \emph{{Bilevel Optimization}}.
\newblock Springer, 2020.

\bibitem[Franceschi et~al.(2018)Franceschi, Frasconi, Salzo, Grazzi, and
  Pontil]{franceschi2018bilevel}
Luca Franceschi, Paolo Frasconi, Saverio Salzo, Riccardo Grazzi, and
  Massimiliano Pontil.
\newblock Bilevel programming for hyperparameter optimization and
  meta-learning.
\newblock In \emph{Proc. Intl. Conf. Machine Learn.}, pages 1568--1577, Vienna,
  Austria, June 2018.

\bibitem[Ghadimi and Lan(2013)]{ghadimi2013sgd}
Saeed Ghadimi and Guanghui Lan.
\newblock Stochastic first-and zeroth-order methods for nonconvex stochastic
  programming.
\newblock \emph{SIAM Journal on Optimization}, 23\penalty0 (4):\penalty0
  2341--2368, 2013.

\bibitem[Ghadimi and Wang(2018)]{ghadimi2018bilevel}
Saeed Ghadimi and Mengdi Wang.
\newblock Approximation methods for bilevel programming.
\newblock \emph{arXiv preprint:1802.02246}, 2018.

\bibitem[Ghadimi et~al.(2020)Ghadimi, Ruszczynski, and Wang]{ghadimi2020jopt}
Saeed Ghadimi, Andrzej Ruszczynski, and Mengdi Wang.
\newblock A single timescale stochastic approximation method for nested
  stochastic optimization.
\newblock \emph{SIAM Journal on Optimization}, 30\penalty0 (1):\penalty0
  960--979, March 2020.

\bibitem[Grazzi et~al.(2020)Grazzi, Franceschi, Pontil, and
  Salzo]{grazzi2020iteration}
Riccardo Grazzi, Luca Franceschi, Massimiliano Pontil, and Saverio Salzo.
\newblock On the iteration complexity of hypergradient computation.
\newblock In \emph{Proc. Intl. Conf. Machine Learn.}, pages 3748--3758,
  virtual, July 2020.

\bibitem[Guo and Yang(2021)]{guo2021randomized}
Zhishuai Guo and Tianbao Yang.
\newblock Randomized stochastic variance-reduced methods for stochastic bilevel
  optimization.
\newblock \emph{arXiv preprint arXiv:2105.02266}, May 2021.

\bibitem[Hong et~al.(2020)Hong, Wai, Wang, and Yang]{hong2020ac}
Mingyi Hong, Hoi-To Wai, Zhaoran Wang, and Zhuoran Yang.
\newblock A two-timescale framework for bilevel optimization: Complexity
  analysis and application to actor-critic.
\newblock \emph{arXiv preprint:2007.05170}, 2020.

\bibitem[Hu et~al.(2020)Hu, Zhang, Chen, and He]{hu2020}
Yifan Hu, Siqi Zhang, Xin Chen, and Niao He.
\newblock Biased stochastic gradient descent for conditional stochastic
  optimization.
\newblock \emph{arXiv preprint:2002.10790}, February 2020.

\bibitem[Huang and Huang(2021)]{huang2021biadam}
Feihu Huang and Heng Huang.
\newblock Biadam: Fast adaptive bilevel optimization methods.
\newblock \emph{arXiv preprint arXiv:2106.11396}, 2021.

\bibitem[Ji et~al.(2020)Ji, Yang, and Liang]{ji2020provably}
Kaiyi Ji, Junjie Yang, and Yingbin Liang.
\newblock Provably faster algorithms for bilevel optimization and applications
  to meta-learning.
\newblock \emph{arXiv preprint:2010.07962}, 2020.

\bibitem[Kaledin et~al.(2020)Kaledin, Moulines, Naumov, Tadic, and
  Wai]{kaledin2020finite}
Maxim Kaledin, Eric Moulines, Alexey Naumov, Vladislav Tadic, and Hoi-To Wai.
\newblock Finite time analysis of linear two-timescale stochastic approximation
  with markovian noise.
\newblock In \emph{Conference on Learning Theory}, pages 2144--2203, Virtual,
  July 2020.

\bibitem[Khanduri et~al.(2021)Khanduri, Zeng, Hong, Wai, Wang, and
  Yang]{khanduri2021momentum}
Prashant Khanduri, Siliang Zeng, Mingyi Hong, Hoi-To Wai, Zhaoran Wang, and
  Zhuoran Yang.
\newblock A momentum-assisted single-timescale stochastic approximation
  algorithm for bilevel optimization.
\newblock \emph{arXiv preprintarXiv:2102.07367}, February 2021.

\bibitem[Konda and Borkar(1999)]{konda1999actor}
Vijaymohan Konda and Vivek Borkar.
\newblock Actor-critic-type learning algorithms for markov decision processes.
\newblock \emph{SIAM Journal on Control and Optimization}, 38\penalty0
  (1):\penalty0 94--123, 1999.

\bibitem[Kunapuli et~al.(2008)Kunapuli, Bennett, Hu, and Pang]{kunapuli2008}
Gautam Kunapuli, Kristin~P Bennett, Jing Hu, and Jong-Shi Pang.
\newblock Classification model selection via bilevel programming.
\newblock \emph{Optimization Methods \& Software}, 23\penalty0 (4):\penalty0
  475--489, 2008.

\bibitem[Kunisch and Pock(2013)]{kunisch2013bilevel}
Karl Kunisch and Thomas Pock.
\newblock A bilevel optimization approach for parameter learning in variational
  models.
\newblock \emph{SIAM Journal on Imaging Sciences}, 6\penalty0 (2):\penalty0
  938--983, 2013.

\bibitem[Lian et~al.(2017)Lian, Wang, and Liu]{lian2017aistats}
Xiangru Lian, Mengdi Wang, and Ji~Liu.
\newblock Finite-sum composition optimization via variance reduced gradient
  descent.
\newblock In \emph{Proc. Intl. Conf. on Artif. Intell. and Stat.}, Fort
  Lauderdale, FL, April 2017.

\bibitem[Lin et~al.(2020)Lin, Jin, and Jordan]{lin2020gradient}
Tianyi Lin, Chi Jin, and Michael Jordan.
\newblock On gradient descent ascent for nonconvex-concave minimax problems.
\newblock In \emph{Proc. Intl. Conf. Machine Learn.}, pages 6083--6093,
  Virtual, July 2020.

\bibitem[Liu et~al.(2020)Liu, Mu, Yuan, Zeng, and Zhang]{liu2020icml}
Risheng Liu, Pan Mu, Xiaoming Yuan, Shangzhi Zeng, and Jin Zhang.
\newblock A generic first-order algorithmic framework for bi-level programming
  beyond lower-level singleton.
\newblock In \emph{Proc. Intl. Conf. Machine Learn.}, pages 6305--6315,
  Virtual, July 2020.

\bibitem[Liu et~al.(2021)Liu, Gao, Zhang, Meng, and Lin]{liu2021investigating}
Risheng Liu, Jiaxin Gao, Jin Zhang, Deyu Meng, and Zhouchen Lin.
\newblock Investigating bi-level optimization for learning and vision from a
  unified perspective: A survey and beyond.
\newblock \emph{IEEE Transactions on Pattern Analysis and Machine
  Intelligence}, 2021.

\bibitem[Luo et~al.(2020)Luo, Ye, Huang, and Zhang]{luo2020stochastic}
Luo Luo, Haishan Ye, Zhichao Huang, and Tong Zhang.
\newblock Stochastic recursive gradient descent ascent for stochastic
  nonconvex-strongly-concave minimax problems.
\newblock In \emph{Proc. Advances in Neural Info. Process. Syst.}, Virtual,
  December 2020.

\bibitem[Mokhtari et~al.(2020)Mokhtari, Ozdaglar, and
  Pattathil]{mokhtari2020unified}
Aryan Mokhtari, Asuman Ozdaglar, and Sarath Pattathil.
\newblock A unified analysis of extra-gradient and optimistic gradient methods
  for saddle point problems: Proximal point approach.
\newblock In \emph{Proc. Intl. Conf. on Artif. Intell. and Stat.}, pages
  1497--1507, Palermo, Italy, August 2020.

\bibitem[Nesterov(2013)]{nesterov2013}
Yurii Nesterov.
\newblock \emph{Introductory Lectures on Convex Optimization: {A} basic
  course}, volume~87.
\newblock Springer, Berlin, Germany, 2013.

\bibitem[Nocedal and Wright(2006)]{nocedal2006numerical}
Jorge Nocedal and Stephen Wright.
\newblock \emph{Numerical Optimization}.
\newblock Springer, Berlin, Germany, 2006.

\bibitem[Nouiehed et~al.(2019)Nouiehed, Sanjabi, Huang, Lee, and
  Razaviyayn]{nouiehed2019solving}
Maher Nouiehed, Maziar Sanjabi, Tianjian Huang, Jason~D Lee, and Meisam
  Razaviyayn.
\newblock Solving a class of non-convex min-max games using iterative first
  order methods.
\newblock In \emph{Proc. Advances in Neural Info. Process. Syst.}, pages
  14934--14942, Vancouver, Canada, December 2019.

\bibitem[Paternain et~al.(2019)Paternain, Mokhtari, and
  Ribeiro]{paternain2019newton}
Santiago Paternain, Aryan Mokhtari, and Alejandro Ribeiro.
\newblock A {Newton}-based method for nonconvex optimization with fast evasion
  of saddle points.
\newblock \emph{SIAM Journal on Optimization}, 29\penalty0 (1):\penalty0
  343--368, January 2019.

\bibitem[Pedregosa(2016)]{pedregosa2016hyperparameter}
Fabian Pedregosa.
\newblock Hyperparameter optimization with approximate gradient.
\newblock In \emph{Proc. Intl. Conf. Machine Learn.}, pages 737--746, New York,
  NY, June 2016.

\bibitem[Rafique et~al.(2021)Rafique, Liu, Lin, and Yang]{rafique2018non}
Hassan Rafique, Mingrui Liu, Qihang Lin, and Tianbao Yang.
\newblock Non-convex min-max optimization: Provable algorithms and applications
  in machine learning.
\newblock \emph{Optimization Methods and Software}, March 2021.

\bibitem[Rajeswaran et~al.(2019)Rajeswaran, Finn, Kakade, and
  Levine]{rajeswaran2019meta}
Aravind Rajeswaran, Chelsea Finn, Sham~M Kakade, and Sergey Levine.
\newblock Meta-learning with implicit gradients.
\newblock In \emph{Proc. Advances in Neural Info. Process. Syst.}, pages
  113--124, Vancouver, Canada, December 2019.

\bibitem[Robbins and Monro(1951)]{robbins1951}
Herbert Robbins and Sutton Monro.
\newblock A stochastic approximation method.
\newblock \emph{Annals of Mathematical Statistics}, 22\penalty0 (3):\penalty0
  400--407, September 1951.

\bibitem[Sabach and Shtern(2017)]{sabach2017jopt}
Shoham Sabach and Shimrit Shtern.
\newblock A first order method for solving convex bilevel optimization
  problems.
\newblock \emph{SIAM Journal on Optimization}, 27\penalty0 (2):\penalty0
  640--660, 2017.

\bibitem[Shaban et~al.(2019)Shaban, Cheng, Hatch, and
  Boots]{shaban2019truncated}
Amirreza Shaban, Ching-An Cheng, Nathan Hatch, and Byron Boots.
\newblock Truncated back-propagation for bilevel optimization.
\newblock In \emph{Proc. Intl. Conf. on Artif. Intell. and Stat.}, pages
  1723--1732, Naha, Okinawa, Japan, April 2019.

\bibitem[Shapiro et~al.(2009)Shapiro, Dentcheva, and
  Ruszczy{\'n}ski]{shapiro2009}
Alexander Shapiro, Darinka Dentcheva, and Andrzej Ruszczy{\'n}ski.
\newblock \emph{Lectures on Stochastic Programming: Modeling and Theory}.
\newblock SIAM, Philadelphia, PA, 2009.

\bibitem[Simonetto et~al.(2016)Simonetto, Mokhtari, Koppel, Leus, and
  Ribeiro]{simonetto2016class}
Andrea Simonetto, Aryan Mokhtari, Alec Koppel, Geert Leus, and Alejandro
  Ribeiro.
\newblock A class of prediction-correction methods for time-varying convex
  optimization.
\newblock \emph{IEEE Transactions on Signal Processing}, 64\penalty0
  (17):\penalty0 4576--4591, May 2016.

\bibitem[Stackelberg(1952)]{stackelberg}
Heinrich~Von Stackelberg.
\newblock \emph{The Theory of Market Economy}.
\newblock Oxford University Press, 1952.

\bibitem[Tran-Dinh et~al.(2020)Tran-Dinh, Pham, and Nguyen]{tran2020stochastic}
Quoc Tran-Dinh, Nhan Pham, and Lam Nguyen.
\newblock Stochastic {Gauss}-{Newton} algorithms for nonconvex compositional
  optimization.
\newblock In \emph{Proc. Intl. Conf. Machine Learn.}, pages 9572--9582,
  Virtual, July 2020.

\bibitem[Tran-Dinh et~al.(2021)Tran-Dinh, Pham, Phan, and
  Nguyen]{tran2021hybrid}
Quoc Tran-Dinh, Nhan~H Pham, Dzung~T Phan, and Lam~M Nguyen.
\newblock A hybrid stochastic optimization framework for composite nonconvex
  optimization.
\newblock \emph{Mathematical Programming}, pages 1--67, 2021.

\bibitem[Vicente and Calamai(1994)]{vicente1994bilevel}
Luis~N Vicente and Paul~H Calamai.
\newblock Bilevel and multilevel programming: A bibliography review.
\newblock \emph{Journal of Global optimization}, 5\penalty0 (3):\penalty0
  291--306, 1994.

\bibitem[Wang et~al.(2017{\natexlab{a}})Wang, Fang, and Liu]{wang2017mp}
Mengdi Wang, Ethan~X Fang, and Han Liu.
\newblock Stochastic compositional gradient descent: algorithms for minimizing
  compositions of expected-value functions.
\newblock \emph{Mathematical Programming}, 161\penalty0 (1-2):\penalty0
  419--449, January 2017{\natexlab{a}}.

\bibitem[Wang et~al.(2017{\natexlab{b}})Wang, Liu, and Fang]{wang2017jmlr}
Mengdi Wang, Ji~Liu, and Ethan Fang.
\newblock Accelerating stochastic composition optimization.
\newblock \emph{Journal Machine Learning Research}, 18\penalty0 (1):\penalty0
  3721--3743, 2017{\natexlab{b}}.

\bibitem[Yang et~al.(2021)Yang, Ji, and Liang]{yang2021provably}
Junjie Yang, Kaiyi Ji, and Yingbin Liang.
\newblock Provably faster algorithms for bilevel optimization.
\newblock \emph{arXiv preprint arXiv:2106.04692}, June 2021.

\bibitem[Ye and Zhu(1995)]{ye1995optimality}
Jane Ye and Daoli Zhu.
\newblock Optimality conditions for bilevel programming problems.
\newblock \emph{Optimization}, 33\penalty0 (1):\penalty0 9--27, 1995.

\bibitem[Zhang and Xiao(2019)]{zhang2019nips}
Junyu Zhang and Lin Xiao.
\newblock A stochastic composite gradient method with incremental variance
  reduction.
\newblock In \emph{Proc. Advances in Neural Info. Process. Syst.}, pages
  9075--9085, Vancouver, Canada, December 2019.

\end{thebibliography}
\bibliographystyle{plainnat}

\clearpage
\onecolumn
\aistatstitle{Supplementary Material for\\
``\FullTitle"}

\appendix
\section{Proof sketch}
 In this section, we highlight the key steps of the proof towards Theorem \ref{theorem1}. The proof for the strongly convex case in Theorem \ref{theorem2} will follow similar steps. 

For simplicity of the convergence analysis, we define the following Lyapunov function 
\begin{align*}\label{eq.Lya1}
    \mathbb{V}^k \!:=\! \Phi_{1/\rho}(x^k) &+  \|y^k - y^*(x^k)\|^2 + \|H_{yy}^k - \nabla_{yy}^2g(x^k, y^k)\|^2 +  \|H_{xy}^k - \nabla_{xy}^2g(x^k, y^k)\|^2 \numberthis
\end{align*}
 which mimics the continuous-time Lyapunov function \eqref{eq.cont-lyap} for the deterministic problem.

Similar to the ODE analysis, we first quantify the difference between two Lyapunov functions as
\begin{align*}\label{eq.diff-Lya}
  \mathbb{V}^{k+1}  -  \mathbb{V}^k   =  & \underbracket{ \Phi_{1/\rho}(x^{k+1}) -  \Phi_{1/\rho}(x^k)}_{\rm Lemma~\ref{lemma3}}  
 + \underbracket{ \|y^{k+1} - y^*(x^{k+1})\|^2 - \|y^k - y^*(x^k)\|^2 }_{\rm Lemma~\ref{lemma2}}\\
&  + \underbracket{ \|H_{yy}^{k+1} - \nabla_{yy}^2g(x^{k+1}, y^{k+1})\|^2  -  \|H_{yy}^k - \nabla_{yy}^2g(x^k, y^k)\|^2 }_{\rm Lemma~\ref{scsc-lemma2}}\\
& + \underbracket{ \|H_{xy}^{k+1} - \nabla_{xy}^2g(x^{k+1}, y^{k+1})\|^2 -  \|H_{xy}^k - \nabla_{xy}^2g(x^k, y^k)\|^2}_{\rm Lemma~\ref{scsc-lemma2}}.\numberthis
\end{align*}

The difference in \eqref{eq.diff-Lya} consists of four difference terms: the first term quantifies the descent of the Moreau Envelope of the upper-level objective functions; the second term characterizes the descent of the lower-level optimization errors; and, the third and fourth terms measure the estimation error of the second-order quantities. Since Lemma \ref{lemma3} stated in the main body bounded the Moreau Envelop of the upper level, we will bound the rest, respectively, in the ensuing lemmas. 

We will analyze the error of the lower-level variable, which is the key step to improving the existing results. 
\begin{lemma}[Error of lower level]
\label{lemma2}
Suppose that Assumptions 1--3 hold, and $y^{k+1}$ is generated by running iteration \eqref{eq.STABLE} given $x^k$. If we choose $\beta_k\leq\frac{2}{\mu_g+L_g}$, then $y^{k+1}$ satisfies 
\begin{align*}\label{eq.lemma2}
& \EE\left[\|y^*(x^{k+1})-y^{k+1}\|^2|{\cal F}^k\right] 
\leq \left(1-\frac{\mu_g L_g\beta^k}{\mu_g + L_g}+ \frac{c\alpha_k^2}{\beta_k}\right)\|y^k-y^*(x^k)\|^2   +  \left(1+\frac{\mu_g L_g\beta^k}{\mu_g + L_g}\right)\beta_k^2\sigma_{g_y}^2\\
&\quad   + \frac{c\alpha_k^4}{\beta_k} + \EE\left[\|H_{yy}^k - \nabla_{yy}^2g(x^k,y^k)\|^2|{\cal F}^k\right]\frac{c\alpha_k^2}{\beta_k}   + \EE\left[\|H_{xy}^k - \nabla_{xy}^2g(x^k,y^k)\|^2|{\cal F}^k\right]\frac{c\alpha_k^2}{\beta_k}.\!	  \numberthis
\end{align*}
\end{lemma}

Roughly speaking, Lemma \ref{lemma2} implies that if the stepsizes $\alpha_k^2$ and $\beta_k^2$ and the estimation errors of  $H_{yy}^k$ and $H_{xy}^k$ are decreasing fast enough, the error of $y^{k+1}$ will also decrease.

Since the RHS of both Lemmas \ref{lemma3} and \ref{lemma2} critically depend on the quality of  $H_{yy}^k$ and $H_{xy}^k$, we will next build upon the results in \citep[Lemma 2]{chen2020scsc} to analyze the estimation errors. 
\begin{lemma}[Estimation errors of $H_{xy}^k$ and $H_{yy}^k$]
\label{scsc-lemma2}
Suppose Assumptions 1--3 hold, and $H_{xy}^k$ and $H_{yy}^k$ are generated by running \eqref{eq.STABLE-H}. The mean square error of $H_{xy}^k$ satisfies 
\begin{align}\label{eq.scsc-lemma2-1}
\!\! \EE&\Big[\|H_{xy}^k-\nabla_{xy}^2g(x^k, y^k)\|^2\mid{\cal F}^k\Big] 
\leq(1-\tau_k)^2\|H_{xy}^{k-1}-\nabla_{xy}^2g(x^{k-1}, y^{k-1})\|^2+2\tau_k^2\sigma_{g_{xy}}^2\nonumber\\
&~~+2(1-\tau_k)^2(\bar{L}_{g_{xy}}^2+L_{g_{xy}}^2)\|x^k -x^{k-1}\|^2 +2(1-\tau_k)^2(\bar{L}_{g_{xy}}^2+L_{g_{xy}}^2)\|y^k-y^{k-1}\|^2\!\!
\end{align}
\newpage
where the constants $L_{g_{xy}}, L_{g_{yy}}, \bar{L}_{g_{xy}}, \bar{L}_{g_{yy}}, \sigma_{g_{xy}}, \sigma_{g_{yy}}$ are defined in Assumptions 1 and 3. 
And likewise, the mean square error of $H_{yy}^k$ satisfies 
\begin{align}\label{eq.scsc-lemma2-2}
 \!\!\! \EE&\Big[\|H_{yy}^k-\nabla_{yy}^2g(x^k, y^k)\|^2\mid{\cal F}^k\Big] 
\leq(1-\tau_k)^2\|H_{yy}^{k-1}-\nabla_{yy}^2g(x^{k-1}, y^{k-1})\|^2+2\tau_k^2\sigma_{g_{yy}}^2\nonumber\\
&~~+2(1-\tau_k)^2(\bar{L}_{g_{yy}}^2+L_{g_{yy}}^2)\|x^k-x^{k-1}\|^2 +2(1-\tau_k)^2(\bar{L}_{g_{yy}}^2+L_{g_{yy}}^2)\|y^k-y^{k-1}\|^2\!.\!
\end{align}
\end{lemma}

Intuitively, the update of $x^k$ is bounded and so is the update of $y^k$, and thus $\|x^k-x^{k-1}\|^2={\cal O}(\alpha_{k-1}^2)$ and $\|y^k-y^{k-1}\|^2={\cal O}(\beta_{k-1}^2)$. Plugging them into the RHS of Lemma \ref{scsc-lemma2}, it suggests that if the stepsizes $\alpha_k^2, \beta_k^2, \tau_k^2$ are decreasing, then the estimation errors of $H_{xy}^k$ and $H_{yy}^k$ also decrease.

Applying Lemmas \ref{lemma3}--\ref{scsc-lemma2} to \eqref{eq.diff-Lya} and rearranging terms, we will be able to get
\begin{align*}\label{eq.diff-Lya-all}
 \EE[\mathbb{V}^{k+1}]  -  \EE[\mathbb{V}^k]   \leq    - c_1\EE[\|y^k-y^*(x^k)\|^2]   - c_2\EE[\|\widehat{x}(x^k)-x^k\|^2]+ c_3 \numberthis
\end{align*}
where the constants are $c_1={\cal O}(\beta_k)$, $c_2={\cal O}(\alpha_k)$ and $c_3={\cal O}(\alpha_k^2+\beta_k^2+\tau_k^2)$. By choosing stepsizes $\alpha_k, \beta_k, \tau_k$ as \eqref{eq.stepsize} and telescoping both sides of \eqref{eq.diff-Lya-all}, we obtain the main results in Theorem \ref{theorem1}.

\section{Auxiliary Lemmas}
In this section, we present some auxiliary lemmas that will be used frequently in the proof. 
\begin{lemma}[{\citep[Lemma 2.2]{ghadimi2018bilevel}}]\label{lemma_lip} 
Under Assumptions 1 and 2, we have
\begin{subequations}
		\begin{align}
	\|\overline{\nabla}_x f(x, y^*(x))-\overline{\nabla}_x f(x, y)\|
&	\leq  L_f\|y^*(x)-y\|\\
	\|\nabla F(x_1)-\nabla F(x_2)\|
&	\leq  L_F\|x_1-x_2\|\\
	\|y^*(x_1)-y^*(x_2)\|&\leq  L_y\|x_1-x_2\|
\end{align}
\end{subequations}
and the constants $L_f, L_y, L_F$ are defined as  
\begin{align*}
L_{f}&:=L_{f_x} + \frac{C_{g_{xy}}L_{f_y}}{\mu_g} + \frac{C_{f_y}}{\mu_g}\left(L_{fxy}+\frac{C_{g_{xy}}{L_{g_{yy}}}}{\mu_g}\right),~~~
L_y :=\frac{C_{g_{xy}}}{\mu_g}\\
L_{F}&:=\bar{L}_{f_x} + \frac{C_{g_{xy}}(\bar{L}_{f_y}+L_f)}{\mu_g} + \frac{C_{f_y}}{\mu_g}\left(\bar{L}_{fxy}+\frac{C_{g_{xy}}{\bar{L}_{g_{yy}}}}{\mu_g}\right) 
\end{align*}
where the constants are defined in Assumptions 1--3.
\end{lemma}

 \section{Proof of Proposition \ref{prop1}}
 \begin{proof}
Define the Jacobian matrix
\begin{align*}
  \nabla_xy(x) = 
  \begin{bmatrix}
    \frac{\partial}{\partial x_1} y_1 (x) & \cdots & \frac{\partial}{\partial x_d} y_1 (x)\\
    & \cdots & \\
    \frac{\partial}{\partial x_1} y_{d_y}(x) & \cdots & \frac{\partial}{\partial x_d} y_{d_y} (x)
  \end{bmatrix}.
\end{align*}
By the chain rule, it follows that
\begin{align}\label{grad-deter}
	\nabla F(x):=\nabla_xf\left(x, y^*(x)\right)+\nabla_xy^*(x)^{\top}\nabla_y f\left(x, y^*(x)\right). 
\end{align}
The minimizer $y^*(x)$ satisfies 
\begin{align}
    \nabla_y g(x,y^*(x)) = 0,\quad\text{thus}~~~\nabla_x\left(\nabla_y g(x,y^*(x))\right) = 0.
\end{align}
By the chain rule again, it follows that
\begin{align*}
    \nabla_{xy}^2g\left(x, y^*(x)\right)
    + \nabla_xy^*(x)^{\top}\nabla_{yy}^2 g\left(x, y^*(x)\right) = 0.
\end{align*}
By Assumption 2,  $\nabla_{yy}^2 g\left(x, y^*(x)\right)$ is invertible, so
\begin{align}\label{grad-deter-1}
	\nabla_xy^*(x)^{\top}:=-\nabla_{xy}^2g\left(x, y^*(x)\right)\left[\nabla_{yy}^2 g\left(x, y^*(x)\right)\right]^{-1}. 
\end{align}
By substituting \eqref{grad-deter-1} into \eqref{grad-deter}, we arrive at \eqref{grad-deter-2}. 
\end{proof}

 \section{Proof of Lemma \ref{lemma3}}
 \begin{proof}
Now we turn to analyze the update of $x$. For convenience, we define the update in \eqref{eq.STABLE2} as
\begin{equation}\label{eq.STABLE-h}	
x^{k+1} =\mathcal{P}_{\mathcal{X}}\left(x^k - \alpha_k \bar{h}_f^k\right) ~~~~~~{\rm with}~~~~~~\bar{h}_f^k  := \nabla_xf\left(x^k, y^k;\xi^k\right)-H_{xy}^k(H_{yy}^k)^{-1}\nabla_y f\left(x^k, y^k;\xi^k\right)
\end{equation}
and let $\widehat{x}^{k}$ and $\widehat{x}$ denote $\widehat{x} (x^{k})$ and $\widehat{x} (x)$.

For $\forall x\in\mathcal{X}$, using the weakly convexity of $F$, we know that
\begin{align*}
    F(\hat x)\geq F(x)+\langle\nabla F(x),\hat x-x\rangle+\frac{\mu_F}{2}\|\hat x-x\|^2
\end{align*}
On the other hand, by the definition of $\hat x$, $\forall x\in\mathcal{X}$, it holds that
\begin{align*}
    F(x)+\frac{\rho}{2}\|x-x\|^2-F(\hat x)-\frac{\rho}{2}\|\hat x-x\|^2=F(x)-F(\hat x)-\frac{\rho}{2}\|\hat x-x\|^2\geq 0.
\end{align*}
Adding above two inequalities, we get that
\begin{align*}
    \langle\nabla F(x),\hat x-x\rangle\leq-\frac{\mu_F+\rho}{2}\|\hat x-x\|^2.
\end{align*}
If we choose $\mu_F$ such that $\mu_F+\rho>0$ and using the definition of Moreau Envelop, we have that
\begin{align*}
    \Phi_{1/\rho}(x^{k+1})&=F(\hat x^{k+1})+\frac{\rho}{2}\|x^{k+1}-\hat x^{k+1}\|^2\leq F(\hat x^{k})+\frac{\rho}{2}\|x^{k+1}-\hat x^{k}\|^2\\
    &= F(\hat x^{k})+\frac{\rho}{2}\| x^{k}-\hat x^{k}\|^2+\frac{\rho}{2}\|x^{k+1}-x^{k}\|^2+\rho\langle x^k- \hat x^k,x^{k+1}- x^k\rangle\\
    &= \Phi_{1/\rho}(x^k)+\frac{\rho}{2}\|x^{k+1}-x^{k}\|^2+\rho\langle x^{k+1}- \hat x^k,x^{k+1}- x^k\rangle-\rho\|x^k-x^{k+1}\|^2\\
    &\leq \Phi_{1/\rho}(x^k)+\rho\alpha_k\langle \hat x^k- x^k,h_f^k\rangle+\rho\alpha_k\langle h_f^k, x^k- x^{k+1}\rangle\numberthis\label{eq:hong}
\end{align*}
where the fourth inequality is due to $\langle x^k-\alpha_k h_f^k-x^{k+1}, \hat x^k-x^{k+1}\rangle\leq 0$ using the definition of $\mathcal{P}_\mathcal{X}$. 
Then taking the conditional expectation of both sides in \eqref{eq:hong}, we have that 
\begin{align*}
    \EE\left[\Phi_{1/\rho}(x^{k+1})|\mathcal{F}^k\right]&\leq \Phi_{1/\rho}(x^{k})+\rho \alpha_k\mathbb{E}\left[\left\langle\widehat{x}^{k}-x^{k}, \bar h_{f}^{k}\right\rangle | \mathcal{F}_{k}\right]+\alpha_k^{2} \rho\EE[\left\|\bar h_{f}^{k}\right\|^{2}|\mathcal{F}^k]\\
    &\leq \Phi_{1/\rho}(x^{k})+2 \rho\alpha_k^2\left(C_{f_x}^2 + \left(\frac{C_{g_{xy}}}{\mu_g}\right)^2C_{f_y}^2\right) +\rho \alpha_k\mathbb{E}\left[\left\langle\widehat{x}^{k}-x^{k}, \bar h_{f}^{k}\right\rangle | \mathcal{F}_{k}\right]\numberthis
    \label{phi_descent}
\end{align*}
where the second inequality comes from \eqref{eq.lemma2-pf-9}. Then we bound the third term in \eqref{phi_descent} and get that
\begin{align*}
    \mathbb{E}\left[\left\langle\widehat{x}^{k}-x^{k}, \bar h_{f}^{k}\right\rangle | \mathcal{F}_{k}\right]
    &\leq \mathbb{E}\left[\left\langle\widehat{x}^{k}-x^{k}, \bar h_{f}^{k}-\bar\nabla_x f(x^k,y^k)+\bar\nabla_x f(x^k,y^k)-\nabla F(x^k)+\nabla F(x^k)\right\rangle| \mathcal{F}_{k}\right]\\
    &\leq \left\langle\widehat{x}^{k}-x^{k}, \EE\left[\bar h_{f}^{k}|\mathcal{F}^k\right]-\bar\nabla_x f(x^k,y^k)\right\rangle+\mathbb{E}\left[\left\langle\widehat{x}^{k}-x^{k}, \bar{\nabla}_{x} f(x^{k}, y^{k})-\nabla F(x^{k})\right\rangle| \mathcal{F}_{k}\right]\\
    &~~~+\mathbb{E}\left[\left\langle\widehat{x}^{k}-x^{k}, \nabla F(x^{k})\right\rangle| \mathcal{F}_{k}\right]\\
    &\leq \frac{\gamma_k}{4}\|\widehat{x}^k-x^k\|^2+\frac{\|\nabla_y f(x^k,y^k)\|^2}{\gamma_k} \EE\left[\big\|(H_{yy}^k)^{-1}H_{xy}^k - H_{yy}(x^k,y^k)^{-1}H_{xy}(x^k,y^k)\big\|^2|\mathcal{F}^k\right]\\
    &~~~+\frac{\gamma_k}{4}\|\widehat{x}^k-x^k\|^2+\frac{1}{\gamma_k}\|\bar{\nabla}_{x} f(x^{k}, y^{k})-\nabla F(x^{k})\|^2-\frac{\mu_F+\rho}{2}\|\widehat{x}^k-x^k\|^2\\
    &\leq \frac{\gamma_k}{2}\|\widehat{x}^k-x^k\|^2+\frac{2C_{f_y}^2}{\gamma_k\mu_g^2} \left[\frac{C_{g_{xy}}^2}{\mu_g^2}\EE[\|H_{yy}^k - H_{yy}(x^k,y^k)\|^2|{\cal F}^k] + \EE[\|H_{xy}^k-H_{xy}(x^k,y^k)\|^2|{\cal F}^k] \right]\\
    &~~~+\frac{L_f^2}{\gamma_k}\|y^k-y^*(x^k)\|^2-\frac{\mu_F+\rho}{2}\|\widehat{x}^k-x^k\|^2
\end{align*}
where the third inequality uses Young's inequality with parameter $\gamma_k$, (61) in \citep{hong2020ac} and the fact that
\begin{equation}
	\EE_{\xi^k}[\bar{h}_f^k|{\cal F}^k] = \nabla_xf\left(x^k, y^k\right)-(H_{yy}^k)^{-1}H_{xy}^k\nabla_y f\left(x^k, y^k\right);
\end{equation}
and the last inequality follows the same steps of \eqref{eq.lemma2-pf-6} and Assumption 3. We choose $\gamma_k=\frac{\mu_F+\rho}{2}$, then we get
\begin{align*}
    \mathbb{E}\left[\left\langle\widehat{x}^{k}-x^{k}, \bar h_{f}^{k}\right\rangle | \mathcal{F}_{k}\right]
    &\leq -\frac{\mu_F+\rho}{4}\|\widehat{x}^k-x^k\|^2+\frac{4C_{f_y}^2C_{g_{xy}}^2}{(\mu_F+\rho)\mu_g^4} \EE[\|H_{yy}^k - H_{yy}(x^k,y^k)\|^2|{\cal F}^k]\\
    &~~~+ \frac{4C_{f_y}^2}{(\mu_F+\rho)\mu_g^2}  \EE[\|H_{xy}^k-H_{xy}(x^k,y^k)\|^2|{\cal F}^k]+\frac{2L_f^2}{\mu_F+\rho}\|y^k-y^*(x^k)\|^2\numberthis
    \label{product}
\end{align*}

Plugging \eqref{product} into \eqref{phi_descent} and taking expectation over all the randomness lead to the lemma. 
\end{proof}

\section{Proof of Lemma \ref{lemma2}}
\begin{proof}
We start by decomposing the error of the lower level variable as
\begin{align*}\label{eq.lemma2-pf-1}
&~~~~~\EE\left[\|y^{k+1}-y^*(x^{k+1})\|^2|{\cal F}^k\right]\\
&=\EE\left[\|y^k - \beta_k h_g^k - y^*(x^k) + y^*(x^k) - y^*(x^{k+1}) - (H_{yy}^k)^{-1}(H_{xy}^k)^{\top}(x^{k+1}-x^k)\|^2|{\cal F}^k\right]\\
&\leq (1+\varepsilon)\underbracket{\EE[\|y^k - \beta_k h_g^k - y^*(x^k)\|^2|{\cal F}^k]}_{I_1}\\
&\quad + (1+\varepsilon^{-1})\underbracket{\EE[\|y^*(x^k) - y^*(x^{k+1}) - (H_{yy}^k)^{-1}(H_{xy}^k)^{\top}(x^{k+1}-x^k)\|^2|{\cal F}^k]}_{I_2}.  \numberthis
\end{align*}

The upper bound of $I_1$ can be derived as
\begin{align*}\label{eq.lemma2-pf-2}
I_1&=\|y^k - y^*(x^k)\|^2 - 2\beta_k\EE[\dotp{y^k-y^*(x^k), h_g^k}|{\cal F}^k] + \beta_k^2\EE[\|h_g^k\|^2|{\cal F}^k]\\
& \stackrel{(a)}{\leq}\|y^k-y^*(x^k)\|^2 - 2\beta_k\dotp{y^k-y^*(x^k), \nabla_y g(x^k, y^k)} + \beta_k^2\|\nabla_yg(x^k,y^k)\|^2 + \beta_k^2\sigma_{g_y}^2\\
& \stackrel{(b)}{\leq} \left(1-\frac{2\mu_g L_g}{\mu_g + L_g}\beta^k\right)\|y^k-y^*(x^k)\|^2 + \beta_k\left(\beta_k-\frac{2}{\mu_g + L_g}\right)\|\nabla_y g(x^k,y^k)\|^2 + \beta_k^2\sigma_{g_y}^2\\
& \stackrel{(c)}{\leq} \left(1-\frac{2\mu_g L_g}{\mu_g + L_g}\beta^k\right)\|y^k-y^*(x^k)\|^2 + \beta_k^2\sigma_{g_y}^2\numberthis
\end{align*}
where (a) comes from the fact that $\Var[X]=\EE[X^2]-\EE[X]^2$, (b) follows from the $\mu_g$-strong convexity and $L_g$ smoothness of $g(x, y)$ \citep[Theorem 2.1.11]{nesterov2013}, and (c) follows from the choice of stepsize $\beta_k\leq \frac{\mu_g /L_g}{32(\mu_g+L_g) }\leq \frac{2}{\mu_g+L_g}$ in \eqref{eq.stepsize_1}.

The upper bound of $I_2$ can be derived as
\begin{align*}\label{eq.lemma2-pf-3}
I_2&=\EE\left[\left\|y^*(x^k) - y^*(x^{k+1}) - (H_{yy}^k)^{-1}(H_{xy}^k)^{\top}(x^{k+1}-x^k)\right\|^2|{\cal F}^k\right]\\
&\leq 3\EE\left[\left\|y^*(x^{k+1})-y^*(x^k)-\nabla_xy^*(x^k)(x^{k+1}-x^k)\right\|^2|{\cal F}^k\right]\\
&\quad + 3\EE\left[\left\|\left(\nabla_xy^*(x^k) - H_{yy}(x^k,y^k)^{-1}H_{xy}(x^k,y^k)^{\top}\right)(x^{k+1}-x^k)\right\|^2|{\cal F}^k\right]\\
&\quad + 3\EE\left[\left\|\left(H_{yy}(x^k,y^k)^{-1}H_{xy}(x^k,y^k)^{\top} - (H_{yy}^k)^{-1}(H_{xy}^k)^{\top}\right)(x^{k+1}-x^k)\right\|^2|{\cal F}^k\right]. \numberthis
\end{align*}

We first bound the first approximation error in the RHS of \eqref{eq.lemma2-pf-3} by
\begin{align*}\label{eq.lemma2-pf-4}
 & \left\|y^*(x^{k+1})-y^*(x^k)-\nabla_xy^*(x^k)(x^{k+1}-x^k)\right\|^2 \\
=  & \left\|\int_0^1\nabla_xy^*(x^k+t(x^{k+1}-x^k))(x^{k+1}-x^k)dt-\nabla_xy^*(x^k)(x^{k+1}-x^k)\right\|^2\\
\leq  &\int_0^1\left\|\nabla_xy^*(x^k+t(x^{k+1}-x^k))-\nabla_xy^*(x^k)\right\|^2\|x^{k+1}-x^k\|^2dt 
\leq  \frac{L_y^2}{2}\|x^{k+1}-x^k\|^4 \numberthis
\end{align*}
where the first inequality follows from the Cauchy-Schwarz inequality, and the second inequality follows from the $L_y$-Lipschitz continuity of $\nabla_xy^*(x)$ in Lemma \ref{lemma_lip}.

Next we bound the second term in the RHS of \eqref{eq.lemma2-pf-3} as
\begin{align*}\label{eq.lemma2-pf-3-1}
& \EE\left[\left\|\left(\nabla_xy^*(x^k) - H_{yy}(x^k,y^k)^{-1}H_{xy}(x^k,y^k)^{\top}\right)(x^{k+1}-x^k)\right\|^2|{\cal F}^k\right]\\
 \leq & \EE\left[\left\|\nabla_xy^*(x^k) - H_{yy}(x^k,y^k)^{-1}H_{xy}(x^k,y^k)^{\top}\right\|^2 \left\|x^{k+1}-x^k\right\|^2|{\cal F}^k\right]\numberthis
\end{align*}
and likewise, the third term of \eqref{eq.lemma2-pf-3} as
\begin{align*}\label{eq.lemma2-pf-3-2}
& \EE\left[\left\|\left(H_{yy}(x^k,y^k)^{-1}H_{xy}(x^k,y^k)^{\top} - (H_{yy}^k)^{-1}(H_{xy}^k)^{\top}\right)(x^{k+1}-x^k)\right\|^2|{\cal F}^k\right]\\
 \leq & \EE\left[\left\|H_{yy}(x^k,y^k)^{-1}H_{xy}(x^k,y^k)^{\top} - (H_{yy}^k)^{-1}(H_{xy}^k)^{\top}\right\|^2 \left\|x^{k+1}-x^k\right\|^2|{\cal F}^k\right].\numberthis
\end{align*}

We then bound the approximation error of $H_{yy}(x^k,y^k)^{-1}H_{xy}(x^k,y^k)^{\top}$ in \eqref{eq.lemma2-pf-3-1} by
\begin{align*}\label{eq.lemma2-pf-5-1}
 & \left\|\nabla_xy^*(x^k) - H_{yy}(x^k,y^k)^{-1}H_{xy}(x^k,y^k)^{\top}\right\|^2\\
\!\!=  & \left\|H_{yy} \left(x^k, y^*(x^k)\right)^{-1}H_{xy} \left(x^k, y^*(x^k)\right)^{\top} - H_{yy}(x^k,y^k)^{-1}H_{xy}(x^k,y^k)^{\top}\right\|^2 \\
\!\!=  & \Big\|H_{yy} \left(x^k, y^*(x^k)\right)^{-1}H_{xy} \left(x^k, y^*(x^k)\right)^{\top} - H_{yy}(x^k,y^k)^{-1}H_{xy} \left(x^k, y^*(x^k)\right)^{\top} \\
\!\!  &\qquad + H_{yy}(x^k,y^k)^{-1}H_{xy} \left(x^k, y^*(x^k)\right)^{\top} - H_{yy}(x^k,y^k)^{-1}H_{xy}(x^k,y^k)^{\top}\Big\|^2\\
\!\!\leq  & 2C_{g_{xy}}^2\Big\|H_{yy} \left(x^k, y^*(x^k)\right)^{-1} \!\!-\! H_{yy}(x^k,y^k)^{-1} \Big\|^2 +\frac{2}{\mu_g^2}\Big\|  H_{xy} \left(x^k, y^*(x^k)\right) \!-\!  H_{xy}(x^k,y^k)\Big\|^2\numberthis
\end{align*}
where the inequality follows from $\|H_{xy}(x,y)\|\leq C_{g_{xy}}$ and $H_{yy}(x,y)\succeq\mu_g I$.

Note that 
\begin{align*}\label{eq.lemma2-pf-5-2}
&\Big\|H_{yy} \left(x^k, y^*(x^k)\right)^{-1} \!-\! H_{yy}(x^k,y^k)^{-1} \Big\|^2\\
=&\Big\|H_{yy} \left(x^k, y^*(x^k)\right)^{-1} \Big(H_{yy} \left(x^k, y^*(x^k)\right)\!-\! H_{yy}(x^k,y^k)\Big) H_{yy}(x^k,y^k)^{-1} \Big\|^2\\
\leq&\Big\|H_{yy} \left(x^k, y^*(x^k)\right)^{-1} \Big\|^2\Big\| H_{yy} \left(x^k, y^*(x^k)\right)\!-\! H_{yy}(x^k,y^k)\Big\|^2 \Big\|H_{yy}(x^k,y^k)^{-1} \Big\|^2\\
\leq  & \frac{1}{\mu_g^4}\Big\|H_{yy} \left(x^k, y^*(x^k)\right)- H_{yy}(x^k,y^k)\Big\|^2  \numberthis
\end{align*}
where the last inequality follows from $H_{yy}(x,y)\succeq\mu_g I$.

Therefore, we have
\begin{align*}\label{eq.lemma2-pf-5}
 & \left\|\nabla_xy^*(x^k) - H_{yy}(x^k,y^k)^{-1}H_{xy}(x^k,y^k)^{\top}\right\|^2\\
 \leq  & \frac{2C_{g_{xy}}^2}{\mu_g^4}\Big\|H_{yy} \left(x^k, y^*(x^k)\right)- H_{yy}(x^k,y^k)\Big\|^2  +\frac{2}{\mu_g^2}\Big\|  H_{xy} \left(x^k, y^*(x^k)\right) -  H_{xy}(x^k,y^k)\Big\|^2.\numberthis
\end{align*}

Following the steps towards \eqref{eq.lemma2-pf-5}, we bound the error of $(H_{yy}^k)^{-1}(H_{xy}^k)^{\top}$ in \eqref{eq.lemma2-pf-3-2} by
\begin{align}\label{eq.lemma2-pf-6}
 & \left\|(H_{yy}^k)^{-1}(H_{xy}^k)^{\top} - H_{yy}(x^k,y^k)^{-1}H_{xy}(x^k,y^k)^{\top}\right\|^2\nonumber\\
 =&\Big\|(H_{yy}^k)^{-1}(H_{xy}^k)^{\top}\!-\!H_{yy}(x^k,y^k)^{-1}(H_{xy}^k)^{\top}\!+\!H_{yy}(x^k,y^k)^{-1}(H_{xy}^k)^{\top} \!-\! H_{yy}(x^k,y^k)^{-1}H_{xy}(x^k,y^k)^{\top}\Big\|^2\nonumber\\
 \leq &2\Big\|(H_{yy}^k)^{-1}(H_{xy}^k)^{\top}\!-\!H_{yy}(x^k,y^k)^{-1}(H_{xy}^k)^{\top}\Big\|^2+2\Big\|H_{yy}(x^k,y^k)^{-1}(H_{xy}^k)^{\top} \!-\! H_{yy}(x^k,y^k)^{-1}H_{xy}(x^k,y^k)^{\top}\Big\|^2\nonumber\\
 \leq &\frac{2C_{g_{xy}}^2}{\mu_g^4}\Big\|H_{yy}^k - H_{yy}(x^k,y^k)\Big\|^2+ \frac{2}{\mu_g^2}\Big\|H_{xy}^k-H_{xy}(x^k,y^k)\Big\|^2
\end{align}
where the second inequality follows from $\|H_{xy}^k\|\leq C_{g_{xy}}$ and $H_{yy}^k\succeq\mu_g I$.

Plugging \eqref{eq.lemma2-pf-4}-\eqref{eq.lemma2-pf-6} back to \eqref{eq.lemma2-pf-3}, we have
\begin{align*}\label{eq.lemma2-pf-7-0}
I_2&\leq \frac{3L_y^2}{2}\EE[\|x^{k+1}-x^k\|^4|{\cal F}^k] +  \frac{6C_{g_{xy}}^2}{\mu_g^4}\|H_{yy}(x^k,y^*(x^k))-H_{yy}(x^k,y^k)\|^2\EE[\|x^{k+1}-x^k\|^2|{\cal F}^k]\\
&\quad + \frac{6}{\mu_g^2}\|H_{xy}(x^k,y^*(x^k))-H_{xy}(x^k,y^k)\|^2 \EE[\|x^{k+1}-x^k\|^2|{\cal F}^k]\\
&\quad + \frac{6C_{g_{xy}}^2}{\mu_g^4}\EE[\|H_{yy}^k - H_{yy}(x^k,y^k)\|^2\|x^{k+1}-x^k\|^2|{\cal F}^k] \\
&\quad+ \frac{6}{\mu_g^2}\EE[\|H_{xy}^k-H_{xy}(x^k,y^k)\|^2\|x^{k+1}-x^k\|^2|{\cal F}^k]. \numberthis
\end{align*}

Using the Lipschitz continuity of $H_{xy}(x,y)$ and $H_{yy}(x,y)$ in Assumption 1, from \eqref{eq.lemma2-pf-7-0}, we have
\begin{align*}\label{eq.lemma2-pf-7}
I_2&\leq \frac{3L_y^2}{2}\EE[\|x^{k+1}-x^k\|^4|{\cal F}^k] + \frac{6}{\mu_g^2}\left(\frac{C_{g_{xy}}^2L_{g_{yy}}}{\mu_g^2}+L_{g_{xy}}\right)\|y^k-y^*(x^k)\|^2\EE[\|x^{k+1}-x^k\|^2|{\cal F}^k]\\
&\quad +  \frac{6C_{g_{xy}}^2}{\mu_g^4}\EE[\|H_{yy}^k - H_{yy}(x^k,y^k)\|^2\|x^{k+1}-x^k\|^2|{\cal F}^k]\\
&\quad +  \frac{6}{\mu_g^2}\EE[\|H_{xy}^k-H_{xy}(x^k,y^k)\|^2\|x^{k+1}-x^k\|^2|{\cal F}^k]. \numberthis
\end{align*}

For any $p=2, 4$, we next analyze quantity $\EE[\|x^{k+1}-x^k\|^p|{\cal F}^k]$ in \eqref{eq.lemma2-pf-7}. Recall the simplified update \eqref{eq.STABLE-h}.
Therefore, we have $ \|x^{k+1}-x^k\| \leq \alpha_k\|\bar{h}_f^k\|$ and
\begin{align*}\label{eq.lemma2-pf-8}
  \|\bar{h}_f^k\| &=\Big\|\nabla_xf\left(x^k, y^k;\xi^k\right)-(H_{yy}^k)^{-1}H_{xy}^k\nabla_y f\left(x^k, y^k;\xi^k\right)\Big\|\\
     &\leq \big\|\nabla_xf(x^k, y^k;\xi^k)\big\|+ \Big\|(H_{yy}^k)^{-1}H_{xy}^k\nabla_y f\left(x^k, y^k;\xi^k\right)\Big\|\\
    &\stackrel{(a)}{\leq} \big\|\nabla_xf(x^k, y^k;\xi^k)\big\|+\frac{C_{g_{xy}}}{\mu_g}\big\|\nabla_y f(x^k, y^k;\xi^k)\big\|  \numberthis
\end{align*}
where (a) follows from the upper and lower projections of $H_{xy}^k$ and $H_{yy}^k$ in \eqref{eq.STABLE-H}.

Therefore, for $p=2, 4$, we have 
\begin{align*}\label{eq.lemma2-pf-9}
    \EE[\|\bar{h}_f^k\|^p|{\cal F}^k,H_{x,y}^k,H_{yy}^k]&\leq 2^{p-1} \EE\Big[\|\nabla_xf(x^k,y^k;\xi^k)\|^p|{\cal F}^k,H_{x,y}^k,H_{yy}^k\Big]\\
    &\quad + 2^{p-1}\left(\frac{C_{g_{xy}}}{\mu_g}\right)^p \EE\Big[\|\nabla_y f(x^k, y^k;\xi^k)\|^p|{\cal F}^k,H_{x,y}^k,H_{yy}^k\Big]\\
    &\leq 2^{p-1}\left(C_{f_x}^p + \left(\frac{C_{g_{xy}}}{\mu_g}\right)^pC_{f_y}^p\right) \numberthis
\end{align*}
where the last inequality from Assumption 3. And thus 
\begin{equation}\label{eq.lemma2-pf-9-2}
        \EE\left[\|x^{k+1}-x^k\|^p|{\cal F}^k,H_{xy}^k,H_{yy}^k\right] \leq 2^{p-1}\left(C_{f_x}^p + \left(\frac{C_{g_{xy}}}{\mu_g}\right)^pC_{f_y}^p\right)\alpha_k^p.
\end{equation}

Plugging \eqref{eq.lemma2-pf-9-2} into \eqref{eq.lemma2-pf-7}, we have
\begin{align*}\label{eq.lemma2-pf-10}
I_2&\leq 12L_y^2 \left(C_{f_x}^4 + \left(\frac{C_{g_{xy}}}{\mu_g}\right)^4C_{f_y}^4\right) \alpha_k^4\\
&\quad + \frac{12}{\mu_g^2}\left(\frac{C_{g_{xy}}^2L_{g_{yy}}}{\mu_g^2}+L_{g_{xy}}\right)\left(C_{f_x}^2 + \left(\frac{C_{g_{xy}}}{\mu_g}\right)^2C_{f_y}^2\right)\|y^k-y^*(x^k)\|^2\alpha_k^2\\
&\quad + \frac{12C_{g_{xy}}^2}{\mu_g^4}\left(C_{f_x}^2 + \left(\frac{C_{g_{xy}}}{\mu_g}\right)^2C_{f_y}^2\right)\EE[\|H_{yy}^k - H_{yy}(x^k,y^k)\|^2|{\cal F}^k]\alpha_k^2\\
&\quad + \frac{12}{\mu_g^2}\left(C_{f_x}^2 + \left(\frac{C_{g_{xy}}}{\mu_g}\right)^2C_{f_y}^2\right)\EE[\|H_{xy}^k-H_{xy}(x^k,y^k)\|^2|{\cal F}^k]\alpha_k^2. \numberthis
\end{align*}

Now let us define the constants as
\begin{align*}
    \tilde{c}_1&:=\max\Bigg\{12L_y^2 \left(C_{f_x}^4 + \left(\frac{C_{g_{xy}}}{\mu_g}\right)^4C_{f_y}^4\right), \frac{12}{\mu_g^2}\left(\frac{C_{g_{xy}}^2L_{g_{yy}}}{\mu_g^2}+L_{g_{xy}}\right)\left(C_{f_x}^2 + \left(\frac{C_{g_{xy}}}{\mu_g}\right)^2C_{f_y}^2\right),\\ 
    &\qquad\qquad \frac{12C_{g_{xy}}^2}{\mu_g^4}\left(C_{f_x}^2 + \left(\frac{C_{g_{xy}}}{\mu_g}\right)^2C_{f_y}^2\right),
    \frac{12}{\mu_g^2}\left(C_{f_x}^2 + \left(\frac{C_{g_{xy}}}{\mu_g}\right)^2C_{f_y}^2\right)\Bigg\}\\
    \tilde{c}_2&:=\frac{2}{\mu_g + L_g}+\frac{\mu_g + L_g}{\mu_g L_g},~~~~c:=\tilde{c}_1\tilde{c}_2.
\end{align*}

Plugging the upper bounds of $I_1$ in \eqref{eq.lemma2-pf-2} and $I_2$ in \eqref{eq.lemma2-pf-10} into \eqref{eq.lemma2-pf-1} with $\epsilon=\frac{\mu_g L_g}{\mu_g + L_g}\beta^k$, we have
\begin{align*}
&~~~~~\EE\left[\|y^{k+1}-y^*(x^{k+1})\|^2|{\cal F}^k\right]\\
&\leq \left(1-\frac{\mu_g L_g}{\mu_g + L_g}\beta^k\right)\|y^k-y^*(x^k)\|^2 +  \left(1+\frac{\mu_g L_g}{\mu_g + L_g}\beta^k\right)\beta_k^2\sigma_{g_y}^2 + \tilde{c}_1\tilde{c}_2\frac{\alpha_k^4}{\beta_k} \\
&\quad + \tilde{c}_1\tilde{c}_2\frac{\alpha_k^2}{\beta_k}\|y^k-y^*(x^k)\|^2 + \tilde{c}_1\tilde{c}_2\EE\left[\|H_{yy}^k - H_{yy}(x^k,y^k)\|^2|{\cal F}^k\right]\frac{\alpha_k^2}{\beta_k} \\
&\quad + \tilde{c}_1\tilde{c}_2\EE\left[\|H_{xy}^k - H_{xy}(x^k,y^k)\|^2|{\cal F}^k\right]\frac{\alpha_k^2}{\beta_k} \numberthis
\end{align*}
where we have used the fact that 
\begin{align*}
	\left(1+\frac{\mu_g L_g}{\mu_g + L_g}\beta^k\right) \left(1-\frac{2\mu_g L_g}{\mu_g + L_g}\beta^k\right) &\leq 	  1-\frac{\mu_g L_g}{\mu_g + L_g}\beta^k \\
	 \left(1+\left(\frac{\mu_g L_g}{\mu_g + L_g}\beta^k\right)^{-1}\right)&\leq \frac{1}{\beta_k}\left(\frac{2}{\mu_g + L_g}+\frac{\mu_g + L_g}{\mu_g L_g}\right)=\frac{\tilde{c}_2}{\beta_k}
\end{align*}
where the last inequality uses $\beta_k\leq \frac{2}{\mu_g+L_g}$ in \eqref{eq.stepsize_1}. 
The proof is complete by defining $c:=\tilde{c}_1\tilde{c}_2$.
\end{proof}

\section{Proof of Lemma \ref{scsc-lemma2}}
\begin{proof}
Recall that $g(x,y)=\EE_{\phi}[g(x,y,\phi)]$. We only have access to the stochastic estimates of $\nabla_{xy}^2g\left(x, y\right), \nabla_{yy}^2 g\left(x, y\right)$, that is
\begin{equation}
  h_{yy}^k(\phi):=\nabla_{yy}^2 g\left(x^k,y^k;\phi\right),~~~~~  h_{xy}^k(\phi):=\nabla_{xy}^2g\left(x^k,y^k;\phi\right). 
\end{equation}
For notational brevity in the analysis, we define
\begin{align}
& H_{xy}(x,y):=\nabla_{xy}^2g\left(x, y\right),~~~~~~~~~~~~ H_{yy}(x,y):=\nabla_{yy}^2 g\left(x, y\right).
\end{align}
and rewrite the update of \eqref{eq.STABLE-H} as
\begin{subequations}\label{app.eq.STABLE-H}
\begin{align}
H_{xy}^k&:={\cal P}_{\{X:\|X\|\leq C_{g_{xy}}\}}\left\{\hat{H}_{xy}^k \right\}~~~{\rm with}	~~~\hat{H}_{xy}^k\!:=\! (1-\tau_k)( H_{xy}^{k-1}\!-\!h_{xy}^{k-1}(\phi^k)) \!+\!h_{xy}^k(\phi^k)                  \label{app.eq.STABLE-H1}\\
H_{yy}^k&:={\cal P}_{\{X:X\succeq \mu_g I\}}\left\{\hat{H}_{yy}^k\right\}~~~{\rm with}	~~~\hat{H}_{yy}^k :=  (1-\tau_k) \big( H_{yy}^{k-1} -h_{yy}^{k-1}(\phi^k)\big)+h_{yy}^k(\phi^k).	 \label{app.eq.STABLE-H3}
\end{align}
\end{subequations}

To analyze the approximation error of $H_{xy}^k$, we decompose it into
\begin{align}\label{eq.scsc-pf-1}
 \EE\Big[\|H_{xy}^k&-H_{xy}(x^k,y^k)\|^2\big|{\cal F}^k\Big]
   \leq\EE\left[\|\hat{H}_{xy}^k-H_{xy}(x^k,y^k)\|^2\big|{\cal F}^k\right]\nonumber\\
    &=\left\|\EE\left[\hat{H}_{xy}^k-H_{xy}(x^k,y^k)|{\cal F}^k\right]\right\|^2 + \sum\limits_{i,j}\Var\left[(\hat{H}_{xy}^k-H_{xy}(x^k,y^k))_{i,j}|{\cal F}^k\right]
\end{align}
where the inequality holds since the projection onto the convex set $\{X:X\succeq \mu_g I\}$ is non-expansive, and the equality comes from the bias-variance decomposition that $\Var[X]=\EE[X^2]-\EE[X]^2$ for any random variable $X$.

We first analyze the bias term in \eqref{eq.scsc-pf-1} by
\begin{align}\label{eq.scsc-pf-1-1}
    &\EE\left[\hat{H}_{xy}^k-H_{xy}(x^k,y^k)|{\cal F}^k\right]\nonumber\\
    \stackrel{\eqref{eq.STABLE-H}}{=}~ &\EE\left[(1-\tau_k) \left( H_{xy}^{k-1}+ h_{xy}^k(\phi^k) - h_{xy}^{k-1}(\phi^k)\right)+\tau_k h_{xy}^k(\phi^k)-H_{xy}(x^k,y^k)|{\cal F}^k\right]\nonumber\\
    =~&(1-\tau_k) \left( H_{xy}^{k-1}+ H_{xy}(x^k,y^k) - H_{xy}(x^{k-1},y^{k-1}) \right)+\tau_k H_{xy}(x^k,y^k)-H_{xy}(x^k,y^k)\nonumber\\
    =~&(1-\tau_k)\left(H_{xy}^{k-1}-H_{xy}(x^{k-1},y^{k-1})\right).
\end{align}
The variance term in \eqref{eq.scsc-pf-1} follows
\begin{align}\label{eq.scsc-pf-1-2}
    &\sum\limits_{i,j}\Var\left[(\hat{H}_{xy}^k-H_{xy}(x^k,y^k))_{i,j}|{\cal F}^k\right]=\sum\limits_{i,j}\Var\left[(\hat{H}_{xy}^k)_{i,j}|{\cal F}^k\right]\nonumber\\
  \stackrel{\eqref{app.eq.STABLE-H1}}{=}  &\sum\limits_{i,j}\Var\left[(1-\tau_k)(h_{xy}^k(\phi^k) - h_{xy}^{k-1}(\phi^k))_{i,j}+\tau_k (h_{xy}^k(\phi^k))_{i,j}|{\cal F}^k\right]\nonumber\\
    \leq&~2(1-\tau_k)^2\sum\limits_{i,j}\Var\left[(h_{xy}^k(\phi^k) - h_{xy}^{k-1}(\phi^k))_{i,j}|{\cal F}^k\right] + 2\tau_k^2\sum\limits_{i,j}\Var\left[(h_{xy}^k(\phi^k))_{i,j}|{\cal F}^k\right]\nonumber\\
     \stackrel{(a)}{\leq}&2(1-\tau_k)^2\EE\left[\|h_{xy}^k(\phi^k) - h_{xy}^{k-1}(\phi^k)\|^2|{\cal F}^k\right] + 2\tau_k^2\sum\limits_{i,j}\Var\left[(h_{xy}^k(\phi^k))_{i,j}|{\cal F}^k\right]\nonumber\\
     \stackrel{(b)}{\leq}&2(1-\tau_k)^2\left(\bar{L}_{g_{xy}}^2+L_{g_{xy}}^2\right)\left(\|x^k-x^{k-1}\|^2+\|y^k-y^{k-1}\|^2\right) + 2\tau_k^2\sigma_{g_{xy}}^2
\end{align}
where (a) uses $\Var[X]\leq\EE[X]^2$ and (b) follows from Assumptions 1 and 3. 

Therefore, plugging \eqref{eq.scsc-pf-1-1} and \eqref{eq.scsc-pf-1-2} into \eqref{eq.scsc-pf-1}, we have
\begin{align*}
    \EE[\|H_{xy}^k-H_{xy}(x^k,y^k)\|^2|{\cal F}^k]&\leq(1-\tau_k)^2\left\|H_{xy}^{k-1}-H_{xy}(x^{k-1},y^{k-1})\right\|^2+ 2\tau_k^2\sigma_{g_{xy}}^2 \\
    &+2(1-\tau_k)^2\left(\bar{L}_{g_{xy}}^2+L_{g_{xy}}^2\right)\left(\|x^k-x^{k-1}\|^2+\|y^k-y^{k-1}\|^2\right). 
\end{align*}

Similarly, we can derive the approximation error of $H_{yy}^k$ as
\begin{align*}
    \EE[\|H_{yy}^k-H_{yy}(x^k,y^k)\|^2|{\cal F}^k]&\leq(1-\tau_k)^2\|H_{yy}^{k-1}-H_{yy}(x^{k-1},y^{k-1})\|^2+ 2\tau_k^2\sigma_{g_{yy}}^2\nonumber\\
    &+2(1-\tau_k)^2\left(\bar{L}_{g_{yy}}^2+L_{g_{yy}}^2\right)\left(\|x^k-x^{k-1}\|^2+\|y^k-y^{k-1}\|^2\right). 
\end{align*}
The proof is then complete.
\end{proof}

\section{Proof of Theorem \ref{theorem1}}
\begin{proof}
Using Lemmas \ref{lemma3}-\ref{scsc-lemma2}, we, respectively, bound the four difference terms in \eqref{eq.diff-Lya} and obtain
\begin{align*}\label{eq.thm1-pf-1}
   \EE[\mathbb{V}^{k+1}] - \EE[\mathbb{V}^k] 
    &\leq - \frac{(\mu_F+\rho)\rho\alpha_k}{4}\EE[\|\widehat{x}(x^k)-x^k\|^2]  - \left(\frac{\mu_g L_g}{\mu_g + L_g}\beta_k - \tilde{c}_1\tilde{c}_2\frac{\alpha_k^2}{\beta_k}- \frac{2L_f^2\rho\alpha_k}{\mu_F+\rho}\right)\EE[\|y^k-y^*(x^k)\|^2] \\
    &\quad -\left(\tau_{k+1} - \tilde{c}_1\tilde{c}_2\frac{\alpha_k^2}{\beta_k} -  \frac{4\rho\alpha_kC_{g_{xy}}^2C_{f_y}^2}{(\mu_F+\rho)\mu_g^4}\right)\EE[\|H_{yy}^k - H_{yy}(x^k,y^k)\|^2]\\
    &\quad -\left(\tau_{k+1} - \tilde{c}_1\tilde{c}_2\frac{\alpha_k^2}{\beta_k} - \frac{4\rho\alpha_kC_{f_y}^2}{(\mu_F+\rho)\mu_g^2}\right)\EE[\|H_{xy}^k-H_{xy}(x^k,y^k)\|^2]\\
    &\quad + 2 \rho\alpha_k^2\left(C_{f_x}^2 + \left(\frac{C_{g_{xy}}}{\mu_g}\right)^2C_{f_y}^2\right) + \left(1+\frac{\mu_g L_g}{\mu_g + L_g}\beta^k\right)\beta_k^2\sigma_{g_y}^2 + \tilde{c}_1\tilde{c}_2\frac{\alpha_k^4}{\beta_k}+ 4\tau_{k+1}^2\sigma_{g_y}^2 \\
    &\quad  + 4(1-\tau_{k+1})^2\tilde{c}_3\alpha_k^2\left(C_{f_x}^2 + \left(\frac{C_{g_{xy}}}{\mu_g}\right)^2C_{f_y}^2\right) + 2(1-\tau_{k+1})^2\tilde{c}_3\EE[\|y^{k+1}-y^k\|^2].\! \numberthis
\end{align*}
where the constant is defined as $\tilde{c}_3:=\bar{L}_{g_{xy}}^2+L_{g_{xy}}^2+\bar{L}_{g_{yy}}^2+L_{g_{yy}}^2$. 

Note that using the $y$-update \eqref{eq.STABLE3}, we also have
\begin{align*}\label{eq.thm1-pf-2}
 \EE[\|y^{k+1}-y^k\|^2] 
&=\EE\big[\big\|\beta_k h_g^k - (H_{yy}^k)^{-1}H_{xy}^k(x^{k+1}-x^k)\big\|^2\big]\\
&\leq 2\beta_k^2\EE\big[\|h_g^k\|^2\big] + 2\EE\big[ \|(H_{yy}^k)^{-1}\|^2\|H_{xy}^k\|^2\big\|x^{k+1}-x^k\big\|^2\big]\\
&\stackrel{(a)}{\leq} 2\beta_k^2\EE\big[\|\nabla_yg(x^k,y^k)\|^2\big]+ 2\beta_k^2\sigma_{g_y}^2  + 2\EE\big[ \|(H_{yy}^k)^{-1}\|^2\|H_{xy}^k\|^2\big\|x^{k+1}-x^k\big\|^2\big]\\
&\stackrel{(b)}{\leq} 2\beta_k^2\EE\big[\|\nabla_yg(x^k,y^k)\|^2\big]+ 2\beta_k^2\sigma_{g_y}^2  + 2\left(\frac{C_{g_{xy}}}{\mu_g}\right)^2\EE\big[  \big\|x^{k+1}-x^k\big\|^2\big]\\
&\stackrel{(c)}{\leq} 4\beta_k^2L_g^2\EE[\|y^k-y^*(x^k)\|^2] + 2\beta_k^2\sigma_{g_y}^2 + 2\left(\frac{C_{g_{xy}}}{\mu_g}\right)^2\EE[\|x^{k+1}-x^k\|^2] \numberthis
\end{align*}
where (a) follows from $\EE[X^2] =\Var[X]+\EE[X]^2$ and Assumption 3, (b) uses the upper and lower projections of $H_{xy}^k$ and $H_{yy}^k$ in \eqref{eq.STABLE-H}, and (c) is due to $\nabla_yg(x^k,y^*(x^k))=0$ as well as Assumption 1.

Selecting parameter $\tau_k=\frac{1}{\sqrt{K}}$, using \eqref{eq.lemma2-pf-9-2} to bound $\EE[\|x^{k+1}-x^k\|^2]$ and using \eqref{eq.thm1-pf-1}-\eqref{eq.thm1-pf-2}, 
we have
\begin{align*}\label{eq.thm1-pf-3}
  \EE[\mathbb{V}^{k+1}] - \EE[\mathbb{V}^k] 
    &\leq - \frac{(\mu_F+\rho)\rho\alpha_k}{4}\EE[\|\widehat{x}(x^k)-x^k\|^2]  +  \alpha_k^2\left(2\rho + 4\tilde{c}_3 +  8\tilde{c}_3\left(\frac{C_{g_{xy}}}{\mu_g}\right)^2\right)\left(C_{f_x}^2 + \left(\frac{C_{g_{xy}}}{\mu_g}\right)^2C_{f_y}^2\right)\\
    & \quad - \left(\frac{\mu_g L_g}{\mu_g + L_g}\beta_k - \tilde{c}_1\tilde{c}_2\frac{\alpha_k^2}{\beta_k}- \frac{2L_f^2\rho\alpha_k}{\mu_F+\rho}-8\tilde{c}_3\beta_k^2L_g^2\right)\EE[\|y^k-y^*(x^k)\|^2] \\
    &\quad -\left(\frac{1}{\sqrt{K}} - \tilde{c}_1\tilde{c}_2\frac{\alpha_k^2}{\beta_k} - \frac{4\rho C_{g_{xy}}^2C_{f_y}^2\alpha_k}{(\mu_F+\rho)\mu_g^4}\right)\EE[\|H_{yy}^k - H_{yy}(x^k,y^k)\|^2]\\
    &\quad -\left(\frac{1}{\sqrt{K}} - \tilde{c}_1\tilde{c}_2\frac{\alpha_k^2}{\beta_k} - \frac{4\rho C_{f_y}^2\alpha_k}{(\mu_F+\rho)\mu_g^2}\right)\EE[\|H_{xy}^k-H_{xy}(x^k,y^k)\|^2]\\
    &\quad + \left(1+\frac{\mu_g L_g}{\mu_g + L_g}\beta^k\right)\beta_k^2\sigma_{g_y}^2 + \tilde{c}_1\tilde{c}_2\frac{\alpha_k^4}{\beta_k} + \frac{4\sigma_{g_y}^2}{K} + 4\tilde c_3\beta_k^2\sigma_{g_y}^2.\numberthis
\end{align*}

Choosing the stepsize $\alpha_k$ as 
\eqref{eq.stepsize},  
it will lead to  (cf. $c:=\tilde{c}_1\tilde{c}_2$)
\begin{subequations}\label{eq.coeff}
    \begin{align}
   \frac{1}{\sqrt{K}} - \tilde{c}_1\tilde{c}_2\frac{\alpha_k^2}{\beta_k} - \frac{4\rho C_{g_{xy}}^2C_{f_y}^2\alpha_k}{(\mu_F+\rho)\mu_g^4} &\stackrel{(a)}{\geq}  \frac{1}{\sqrt{K}} - \tilde{c}_1\tilde{c}_2 \alpha_k  - \frac{4\rho C_{g_{xy}}^2C_{f_y}^2\alpha_k}{(\mu_F+\rho)\mu_g^4}\stackrel{(b)}{\geq} 0\\
     \frac{1}{\sqrt{K}} - \tilde{c}_1\tilde{c}_2\frac{\alpha_k^2}{\beta_k} - \frac{4\rho C_{f_y}^2\alpha_k}{(\mu_F+\rho)\mu_g^2}&\stackrel{(c)}{\geq}  \frac{1}{\sqrt{K}} - \tilde{c}_1\tilde{c}_2 \alpha_k  - \frac{4\rho C_{f_y}^2\alpha_k}{(\mu_F+\rho)\mu_g^2}\stackrel{(d)}{\geq} 0
\end{align}
where both (a) and (c) follow from $\alpha_k\leq\beta_k$ in \eqref{eq.stepsize_2}; and (b) and (d) follow from the second and the third terms in \eqref{eq.stepsize_2}. In addition, choosing the stepsize $\beta_k$ as \eqref{eq.stepsize} will lead to
\begin{align}
    \frac{\mu_g L_g}{\mu_g + L_g}\beta_k - \tilde{c}_1\tilde{c}_2\frac{\alpha_k^2}{\beta_k}- \frac{2L_f^2\rho\alpha_k}{\mu_F+\rho}-8\tilde c_3\beta_k^2L_g^2&\stackrel{(e)}{\geq}  \frac{\mu_g L_g}{\mu_g + L_g}\beta_k - (\tilde{c}_1\tilde{c}_2+\frac{2L_f^2\rho}{\mu_F+\rho}) \alpha_k -8\tilde c_3\beta_k^2L_g^2\nonumber\\
    &\stackrel{(f)}{\geq} \frac{\mu_g L_g\beta_k }{2(\mu_g+L_g)}-8\tilde c_3\beta_k^2L_g^2\stackrel{(g)}{\geq} \frac{\mu_g L_g\beta_k}{4(\mu_g + L_g)} 
 \end{align} 
\end{subequations}
where (e) follows from $\alpha_k\leq\beta_k$ in \eqref{eq.stepsize_2}, (f) is due to the last terms in \eqref{eq.stepsize_2}, and (g) uses \eqref{eq.stepsize_1}.

Using \eqref{eq.coeff} to cancel terms in \eqref{eq.thm1-pf-3}, we are able to get
\begin{align*}\label{eq.thm1-pf-4}
    \EE[\mathbb{V}^{k+1}]  - \EE[\mathbb{V}^k]\leq - \frac{\mu_g L_g\beta_k}{4(\mu_g + L_g)}\EE[\|y^k-y^*(x^k)\|^2]  - \frac{(\mu_F+\rho)\rho\alpha_k}{4}\EE[\|\widehat{x}(x^k)-x^k\|^2] + {\cal O}\left(\frac{1}{K}\right)  \numberthis
\end{align*}
 from which we can reach Theorem \ref{theorem1} after telescoping the both sides of \eqref{eq.thm1-pf-4}. 
 \end{proof}

 \section{Proof of Theorem \ref{theorem2}}


Slightly different from the Lyapunov function \eqref{eq.Lya}, we define the following Lyapunov function
\begin{align*}
    \mathbb{V}^k := \|x^k-x^*\|^2 + \|y^k - y^*(x^k)\|^2 +\|H_{yy}^k - \nabla_{yy}^2g(x^k, y^k)\|^2 + \|H_{xy}^k - \nabla_{xy}^2g(x^k, y^k)\|^2.
\end{align*}

 \begin{lemma}
\label{lemma4}
Suppose Assumptions 1--3 hold and $F(x)$ is $\mu$-strongly convex. Then $x^k$ satisfies 
\begin{align*}\label{eq.lemma4}
\EE[\|x^{k+1}-x^*\|^2] 
&\leq (1-\mu\alpha_k)\EE[\|x^k-x^*\|^2] + \frac{2L_f^2}{\mu}\alpha_k\EE[\|y^k-y^*(x^k)\|^2] + \alpha_k^2\EE[\|\bar{h}_f^k\|^2]\\
      & + \frac{4C_{g_{xy}}^2C_{f_y}^2}{\mu_g^4\mu}\alpha_k\EE[\|H_{yy}^k - H_{yy}(x^k,y^k)\|^2]  + \frac{4C_{f_y}^2}{\mu_g^2\mu}\alpha_k\EE[\|H_{xy}^k-H_{xy}(x^k,y^k)\|^2]\numberthis
\end{align*}
where $L_f, L_F$ are defined in Lemma \ref{lemma_lip}, and $C_{g_{xy}}$ is the projection radius of $H_{xy}^k$ in \eqref{eq.STABLE-H1}. 
\end{lemma}

\begin{proof}
We start with
\begin{align*}\label{eq.lemma4-pf-1}
    \EE[\|x^{k+1}-x^*\|^2|{\cal F}^k]
    &\stackrel{(a)}{\leq}\EE[\|x^k-\alpha_k\bar{h}_f^k-x^*\|^2|{\cal F}^k]\\
    &=\|x^k-x^*\|^2 - 2\alpha_k\dotp{x^k-x^*, \EE[\bar{h}_f^k|{\cal F}^k]} + \alpha_k^2\EE[\|\bar{h}_f^k\|^2|{\cal F}^k]\\
    &=\|x^k-x^*\|^2 - 2\alpha_k\dotp{x^k-x^*, \nabla F(x^k)} \\
    &\quad+ 2\alpha_k\dotp{x^k-x^*, \nabla F(x^k)-\EE[\bar{h}_f^k|{\cal F}^k]} + \alpha_k^2\EE[\|\bar{h}_f^k\|^2|{\cal F}^k]\\
    &\stackrel{(b)}{\leq}\|x^k-x^*\|^2 - 2\alpha_k\dotp{x^k-x^*, \nabla F(x^k)-\nabla F(x^*)}\\
      &\quad + 2\alpha_k\dotp{x^k-x^*, \nabla F(x^k)-\EE[\bar{h}_f^k|{\cal F}^k]} + \alpha_k^2\EE[\|\bar{h}_f^k\|^2|{\cal F}^k] \numberthis
\end{align*}
where (a) follows the fact that ${\cal P}_{\cal X}$ is non-expansive, and (b) follows the optimality condition that $\dotp{\nabla F(x^*), x-x^*}\geq 0$ for any $x\in{\cal X}$. 

Using the $\mu$-strong convexity of $F(x)$, it follows that 
\begin{equation}
    - \dotp{x^k-x^*, \nabla F(x^k)-\nabla F(x^*)}\leq -\mu \|x^k-x^*\|^2
\end{equation}
plugging which into \eqref{eq.lemma4-pf-1} leads to
\begin{align*} 
    \EE\left[\|x^{k+1}-x^*\|^2|{\cal F}^k\right]
    & \leq (1-2\mu\alpha_k)\|x^k-x^*\|^2 + 2\alpha_k\dotp{x^k-x^*, \nabla F(x^k)-\EE[\bar{h}_f^k|{\cal F}^k]} + \alpha_k^2\EE\left[\|\bar{h}_f^k\|^2|{\cal F}^k\right]\\
    &\stackrel{(c)}{\leq} (1-\mu\alpha_k)\|x^k-x^*\|^2 + \frac{\alpha_k}{\mu}\left\|\nabla F(x^k)-\EE[\bar{h}_f^k|{\cal F}^k]\right\|^2 + \alpha_k^2\EE\left[\|\bar{h}_f^k\|^2|{\cal F}^k\right] \numberthis
\end{align*}
where (c) uses the Young's inequality. 

The approximation error of $\bar{h}_f^k$ can be bounded by
\begin{align*}\label{eq.lemma3-pf-2}
   &\quad~ \big\|\nabla F(x^k)-\EE[\bar{h}_f^k|{\cal F}^k]\big\|^2\\
   &\leq 2\big\|\nabla F(x^k) - \overline{\nabla}f(x^k,y^k)\big\|^2+ 2\EE\big[\|\overline{\nabla}f(x^k,y^k) - \EE_{\xi^k}[\bar{h}_f^k]\|^2|{\cal F}^k\big]\\
    & \stackrel{(a)}{\leq} 2L_f^2\big\|y^k-y^*(x^k)\big\|^2 + 2\EE\big[\|\overline{\nabla}f(x^k,y^k) - \EE_{\xi^k}[\bar{h}_f^k]\|^2|{\cal F}^k\big]\\
 & \stackrel{(b)}{\leq} 2L_f^2\big\|y^k-y^*(x^k)\big\|^2 + 2\big\|(H_{yy}^k)^{-1}H_{xy}^k - H_{yy}(x^k,y^k)^{-1}H_{xy}(x^k,y^k)\big\|^2\big\|\nabla_y f(x^k,y^k)\big\|^2\\   
    & \stackrel{(c)}{\leq} 2L_f^2\|y^k-y^*(x^k)\|^2 + \frac{4C_{g_{xy}}^2C_{f_y}^2}{\mu_g^4}\EE[\|H_{yy}^k - H_{yy}(x^k,y^k)\|^2|{\cal F}^k] + \frac{4C_{f_y}^2}{\mu_g^2}\EE[\|H_{xy}^k-H_{xy}(x^k,y^k)\|^2|{\cal F}^k] \numberthis
\end{align*}
where (a) follows from Lemma \ref{lemma_lip}, (b) uses the fact that 
\begin{equation}
	\EE_{\xi^k}[\bar{h}_f^k|{\cal F}^k] = \nabla_xf\left(x^k, y^k\right)-(H_{yy}^k)^{-1}H_{xy}^k\nabla_y f\left(x^k, y^k\right)
\end{equation}
and (c) follows the same steps of \eqref{eq.lemma2-pf-6} and Assumption 3. Plugging \eqref{eq.lemma3-pf-2} into the above completes the proof.
\end{proof}

Similar to \eqref{eq.diff-Lya}, we first quantify the difference between consecutive Lyapunov functions as
\begin{align*}\label{eq.diff-Lya2}
  \mathbb{V}^{k+1}  -  \mathbb{V}^k   =  & \underbracket{\|x^{k+1}-x^*\|^2 - \|x^k-x^*\|^2}_{\rm Lemma~\ref{lemma4}}  
 + \underbracket{ \|y^{k+1} - y^*(x^{k+1})\|^2 - \|y^k - y^*(x^k)\|^2 }_{\rm Lemma~\ref{lemma2}}\\
&  + \underbracket{ \|H_{yy}^{k+1} - \nabla_{yy}^2g(x^{k+1}, y^{k+1})\|^2 \!-\! \|H_{yy}^k - \nabla_{yy}^2g(x^k, y^k)\|^2 }_{\rm Lemma~\ref{scsc-lemma2}}\\
&  + \underbracket{ \|H_{xy}^{k+1} - \nabla_{xy}^2g(x^{k+1}, y^{k+1})\|^2 \!-\!  \|H_{xy}^k - \nabla_{xy}^2g(x^k, y^k)\|^2}_{\rm Lemma~\ref{scsc-lemma2}}.\numberthis
\end{align*}

Using Lemmas \ref{lemma2}-\ref{scsc-lemma2} and \ref{lemma4} and defining  $\tilde{c}_3:=\bar{L}_{g_{xy}}^2 + L_{g_{xy}}^2 + \bar{L}_{g_{yy}}^2 + L_{g_{yy}}^2$, we obtain
\begin{align*}\label{eq.themm2-pf-1}
    \EE[\mathbb{V}^{k+1}]-\EE[\mathbb{V}^k]&\leq -\mu\alpha_k\EE[\|x^k-x^*\|^2] -\left(\frac{\mu_g L_g\beta_k}{\mu_g + L_g} - \tilde{c}_1\tilde{c}_2\frac{\alpha_k^2}{\beta_k} - \frac{2L_f^2\alpha_k}{\mu}\right)\EE[\|y^k-y^*(x^k)\|^2]\\
    &\quad -\left(\tau_{k+1} - \tilde{c}_1\tilde{c}_2\frac{\alpha_k^2}{\beta_k}  - \frac{4C_{g_{xy}}^2C_{f_y}^2}{\mu_g^4\mu}\alpha_k\right)\EE[\|H_{yy}^k - H_{yy}(x^k,y^k)\|^2]\\
    &\quad -\left(\tau_{k+1} - \tilde{c}_1\tilde{c}_2\frac{\alpha_k^2}{\beta_k}  - \frac{4C_{f_y}^2}{\mu_g^2\mu}\alpha_k\right)\EE[\|H_{xy}^k-H_{xy}(x^k,y^k)\|^2]\\
    &\quad + \alpha_k^2\EE[\|\bar{h}_f^k\|^2] + \left(1+\frac{\mu_g L_g}{\mu_g + L_g}\beta^k\right)\beta_k^2\sigma_{g_y}^2 +  \tilde{c}_1\tilde{c}_2\frac{\alpha_k^4}{\beta_k}+ 4\tau_{k+1}^2\sigma_{g_y}^2  \\
    &\quad + 2(1-\tau_{k+1})^2\tilde{c}_3\EE[\|x^{k+1}-x^k\|^2] + 2(1-\tau_{k+1})^2\tilde{c}_3\EE[\|y^{k+1}-y^k\|^2]. \numberthis
\end{align*}

Plugging \eqref{eq.thm1-pf-2} and \eqref{eq.lemma2-pf-9} into \eqref{eq.themm2-pf-1}, we have
\begin{align*}\label{eq.themm2-pf-2}
    \EE[\mathbb{V}^{k+1}]-\EE[\mathbb{V}^k]&\leq -\mu\alpha_k\EE[\|x^k-x^*\|^2]+  \underbracket{\Bigg(2+4\tilde{c}_3+8\tilde{c}_3\Big(\frac{C_{g_{xy}}}{\mu_g^2}\Big)^2\Bigg)\! \Bigg(C_{f_x}^2 + \Big(\frac{C_{g_{xy}}}{\mu_g}\Big)^2C_{f_y}^2\Bigg)}_{\tilde{c}_4:=}\alpha_k^2\\
    &\quad -\left(\frac{\mu_g L_g\beta_k}{\mu_g + L_g} - \tilde{c}_1\tilde{c}_2\frac{\alpha_k^2}{\beta_k} - \frac{2L_f^2\alpha_k}{\mu} - 8\beta_k^2L^2L_g^2\right)\EE[\|y^k-y^*(x^k)\|^2]\\
    &\quad -\left(\tau_{k+1} - \tilde{c}_1\tilde{c}_2\frac{\alpha_k^2}{\beta_k} - \frac{4C_{g_{xy}}^2C_{f_y}^2}{\mu_g^4\mu}\alpha_k\right)\EE[\|H_{yy}^k - H_{yy}(x^k,y^k)\|^2]\\
    &\quad -\left(\tau_{k+1} - \tilde{c}_1\tilde{c}_2\frac{\alpha_k^2}{\beta_k} - \frac{4C_{f_y}^2}{\mu_g^2\mu}\alpha_k\right)\EE[\|H_{xy}^k-H_{xy}(x^k,y^k)\|^2]\\
    &\quad + \left(1+\frac{\mu_g L_g}{\mu_g + L_g}\beta^k\right)\beta_k^2\sigma_{g_y}^2 + \tilde{c}_1\tilde{c}_2\frac{\alpha_k^4}{\beta_k} + 4\tau_{k+1}^2\sigma_{g_y}^2 + 8L_g^2\beta_k^2\sigma_{g_y}^2.\numberthis
\end{align*}

We choose the stepsizes $\alpha_k, \beta_k, \tau_k$ as \eqref{eq.stepsize2} 
to guarantee that (cf. $c:=\tilde{c}_1\tilde{c}_2$)
\begin{align}\label{eq.coeff2}
&\mathsf{(a)}~~~\tau_{k+1} - \tilde{c}_1\tilde{c}_2\frac{\alpha_k^2}{\beta_k} -\frac{4C_{g_{xy}}^2C_{f_y}^2}{\mu_g^4\mu}\alpha_k\geq  \frac{\beta_k}{4};~~~\mathsf{(b)}~~~
\tau_{k+1} - \tilde{c}_1\tilde{c}_2\frac{\alpha_k^2}{\beta_k} - \frac{4C_{f_y}^2}{\mu_g^2\mu}\alpha_k\geq \frac{\beta_k}{4}\nonumber\\
&\mathsf{(c)}~~~\frac{\mu_g L_g}{\mu_g + L_g}\beta_k - \tilde{c}_1\tilde{c}_2\frac{\alpha_k^2}{\beta_k} - \frac{2L_f^2}{\mu}\alpha_k - 8\beta_k^2L^2L_g^2\geq \frac{\mu_g L_g}{4(\mu_g + L_g)}.
\end{align}

Therefore, plugging \eqref{eq.coeff2} into \eqref{eq.themm2-pf-2}, we have
\begin{align*}
    \EE[\mathbb{V}^{k+1}]-\EE[\mathbb{V}^k]&\leq -\mu\alpha_k\EE\Big[\|x^k-x^*\|^2\Big] - \frac{\mu_g L_g}{4(\mu_g + L_g)}\beta_k\EE\Big[\|y^k-y^*(x^k)\|^2\Big]\\
    &\quad  - \frac{\beta_k}{4}\EE\Big[\|H_{yy}^k - H_{yy}(x^k,y^k)\|^2\Big]- \frac{\beta_k}{4}\EE\Big[\|H_{xy}^k - H_{xy}(x^k,y^k)\|^2\Big] + \tilde{c}_6\beta_k^2\\
    &\leq-\tilde{c}_5\beta_k\EE[\mathbb{V}^k] + \tilde{c}_6\beta_k^2\numberthis
\end{align*}
where the first and second inequalities hold since we define
\begin{align*}
&\tilde{c}_5:=\min\left\{\frac{\mu\alpha_k}{\beta_k}, \frac{\mu_gL_g}{4(\mu_g+L_g)}, \frac{1}{4}\right\}={\cal O}(1)\\
&\tilde{c}_6:=\left(1+\frac{\mu_g L_g}{\mu_g + L_g}\beta^k\right)\sigma_{g_y}^2 + \frac{\alpha_k^2}{4\beta_k} + 4\sigma_{g_y}^2 + 8L_g^2\sigma_{g_y}^2+\tilde{c}_4={\cal O}(1).\numberthis
\end{align*}

If we choose $\beta_k=\frac{2}{\tilde{c}_5(K_0+k)}$, where $K_0$ is a sufficiently large constant, then we have
\begin{align*}
\EE[\mathbb{V}^K]&\leq\prod\limits_{k=0}^{K-1}(1-\tilde{c}_5\beta_k)\mathbb{V}^0 + \tilde{c}_6\sum_{k=0}^{K-1}\beta_k^2\prod\limits_{j=k+1}^{k-1}(1-\tilde{c}_5\beta_j)\\
&\leq\frac{(K_0-2)(K_0-1)}{(K_0+K-2)(K_0+K-1)}\mathbb{V}^0 + \frac{\tilde{c}_6}{\tilde{c}_5^2}\sum_{k=0}^{K-1}\frac{4}{(k+K_0)^2}\frac{(k+K_0-1)(k+K_0)}{(K+K_0-2)(K+K_0-1)}\\
&\leq\frac{(K_0-1)^2}{(K_0+K-1)^2}\mathbb{V}^0 + \frac{4\tilde{c}_6 K}{\tilde{c}_5^2(K+K_0-1)^2}\numberthis
\end{align*}
\vspace{-0.2cm}
from which the proof is complete.
\vspace{-0.2cm}

\end{document}